\documentclass{amsart}
\usepackage[utf8]{inputenc}
\usepackage[legalpaper, margin=4cm]{geometry}
\usepackage{amsfonts}
\usepackage{amsmath}
\usepackage{amsthm}
\usepackage{enumitem}
\usepackage{amssymb}
\usepackage{ytableau}
\usepackage{graphicx}
\usepackage{caption}
\usepackage{afterpage}
\usepackage{tikz}
\usepackage{array}
\usepackage{makecell}
\usepackage{diagbox}
\usepackage{multirow}
\usepackage{multicol}
\usepackage{float}
\usepackage{subcaption}

\makeatletter
\newcommand*\bigcdot{\mathpalette\bigcdot@{.5}}
\newcommand*\bigcdot@[2]{\mathbin{\vcenter{\hbox{\scalebox{#2}{$\m@th#1\bullet$}}}}}
\makeatother

\usepackage[sorting = none]{biblatex}

\author{Thomas Bouchet}

\newtheorem{lemma}{Lemma}
\newtheorem{prop}{Proposition}
\newtheorem{theorem}{Theorem}
\newtheorem*{theorem*}{Theorem}
\newtheorem{cor}{Corollary}
\newtheorem*{cor*}{Corollary}
\theoremstyle{definition}
\newtheorem{definition}{Definition}

\newtheorem{rem}{Remark}
\usepackage{colortbl}
\usepackage{listings}
\newcommand{\fonction}[5]{\begin{array}{c|ccc}
#1: & #2 & \longrightarrow & #3 \\
    & #4 & \longmapsto & #5 \end{array}}
\newcommand{\Pic}{\mathrm{Pic}}

\definecolor{mylinkcolor}{rgb}{0.0,0.0,0.75}
\definecolor{myurlcolor}{rgb}{0.0,0.0,0.75}
\usepackage[
	colorlinks, urlcolor=myurlcolor,citecolor=myurlcolor,linkcolor=mylinkcolor,
	pdfauthor={T. Bouchet},
        pdfusetitle
]{hyperref}

\begin{document}

\title{Invariants of genus 4 curves}

\maketitle

\begin{abstract}
    The present paper gives an explicit classification of the isomorphism classes of non-hyperelliptic genus 4 curves over an algebraically closed field of characteristic 0. 
    A non-hyperelliptic genus 4 curve lies on a quadric in $\mathbb{P}^3$ of rank 3 or 4.
    In the case of rank 3, we give a set of 60 invariants which classify the isomorphism classes, and in the case of rank 4, we find 65 invariants. 
    These invariants are defined by transvectants and can be efficiently computed on a given example.
\end{abstract}

\setcounter{tocdepth}{3}
\tableofcontents

\section{Introduction}

Invariant theory dates back to the 19th century, where the polynomial functions of binary forms invariant by change of variables were extensively studied.
Under the impetus of Cayley, Sylvester, Clebsch and Gordan, much work was done
on the subject. Gordan's algorithm~\cite{gordan} enabled them to find complete systems of invariants for binary forms of degrees up to 8.
This algorithm is based on Clebsch's symbolic method~\cite{clebsch} and the transvectant, 
which is Cayley's $\Omega$-process~\cite[Section 5]{olver} for binary forms.
\medskip

This paper is motivated by the problem of finding an explicit classification of the isomorphism classes of non-hyperelliptic genus 4 curves.
This problem was solved for curves of genus up to 3 (we refer the reader to~\cite{igusa}, \cite{shioda} and \cite{dixmier,ohno}),
so it was natural to try to adress the case of genus 4.
\medskip

It is well known that this problem can be solved for hyperelliptic curves of genus $g$
by finding a complete set of invariants for the algebra of invariants of
binary forms of degree $2g+2$~\cite{lerc-ritz}. For those of genus $g=4$, we have thus to
consider the algebra of invariants of binary forms of degree $10$.
A complete system of generators of this algebra was found by Popoviciu and Brouwer~\cite{popoviciu}.
In~\cite{olive-lerc}, a generating set of invariants for the same algebra is constructed using an improved version of Gordan's algorithm.
\medskip

In this work, any non-hyperelliptic curve of genus 4 is canonically embedded in $\mathbb{P}^3$ as the intersection of an irreducible cubic and an irreducible quadric.
In addition, the quadratic form defining the quadric must be of rank 3 or 4. We will refer to these cases as the rank 3 and rank 4 cases.
\medskip

In this paper, we give an explicit classification of the isomorphism classes of non-hyperelliptic genus 4 curves for the rank 3 and rank 4 cases.
We provide a Magma package for the computation of invariants of genus 4 curves~\cite{Git}.
For the hyperelliptic case, we use the invariants given in~\cite{olive-lerc}. 
\medskip

The paper is organized as follows: in Section~\ref{sec:geometry}, we state our problem with our own notation. 
In Section~\ref{sec:tools}, we introduce some technical tools.
Section~\ref{sec:rank4} deals with the rank 4 case, and is subdivided in many subsections to improve readability.
Section~\ref{sec:rank3} focuses on the resolution of the rank 3 case.
\bigskip

Let us now introduce a few notation. Let $K$ be a commutative, algebraically closed field of characteristic 0.
We use the definitions of algebraic groups and rational representations given in~\cite[A.1]{kemper}.
Let $\Gamma$ be an algebraic group, and $V$ a rational representation of $\Gamma$. Let $M\in\Gamma$, $f\in K[V]$, and $v\in V$. We define the action of $\Gamma$ on $K[V]$ to be
\[[M\bigcdot f](v) = f(M^{-1}\bigcdot v)\,.\]

Let $H$ be the algebraic group $\mathrm{SL}_2(K)\times \mathrm{SL}_2(K)$ and $G$
be the algebraic group $H\rtimes
\mathbb{Z}/2\mathbb{Z}$, where $\mathbb{Z}/2\mathbb{Z}$ acts on $H$ by exchanging the two $\mathrm{SL}_2(K)$ factors.
\medskip

Let $m, n\in\mathbb{Z}_{>0}$. We say that \[f = \displaystyle\sum_{\substack{0\leq i \leq m \\ 0\leq j \leq n}}a_{ij}x^iy^{m-i}u^jv^{n-j}\]
is a biform (also referred to as double binary form in~\cite{turnbull}) of bidegree $(m,n)$. Let $W_{m,n}$ be the space of biforms of bidegree $(m,n)$.
We will write $W_m$ for $W_{m,m}$.

The group $H$ acts on $W_{m,n}$ in a natural way: for any $f\in W_{m,n}$, $(M_1,M_2)\in H$, we let
\[\big[(M_1,M_2)\bigcdot f\big]\big((x,y),(u,v)\big):=f\left(M_1^{-1}(x,y), M_2^{-1}(u,v)\right).\]
That action induces an action of $H$ on the coordinate ring $K[W_{m,n}]$: for $I\in K[W_{m,n}]$, $(M_1,M_2)\in H$,
we let
\[\big[(M_1,M_2)\bigcdot I\big](f)=I\big[(M_1^{-1}, M_2^{-1})\bigcdot f\big].\]

The group $\mathbb{Z}/2\mathbb{Z}$ acts on $W_m$ by switching $(x,y)$ and $(u,v)$.
One can see that the corresponding action on the coordinate ring is the exchange of the coefficients $a_{ij}$ and $a_{ji}$.
Hence the group $G$ acts on $W_m$.
\medskip

We say that $f,g\in W_{m,n}$ are $H$-equivalent (resp. $G$-equivalent) if
there exists $M\in H$ (resp. $M\in G$) such that $M\bigcdot f = g$.
\medskip

For clarity, we will often reason with the table of
coefficients of a biform instead of the biform itself.
For any $f\in W_{m,n}$, the $(i,j)$-th coefficient of the table is the coefficient of $x^{m-j}y^ju^{n-i}v^i$ in $f$.
The vanishing coefficients are represented by a blank square, and the non-zero coefficients by a gray square.
\medskip

For example, let $f = 3xu-4yv\in W_1$. Its table representation is\,
\tabcolsep=0.065cm
    \begin{tabular}{|>{\centering}m{.6cm}|>{\centering}m{.6cm}|>{\centering\arraybackslash}m{.6cm}|}
        \cline{2-3}
        \multicolumn{1}{c|}{}& $x$ & $y$ \\
        \hline
        $u$  & \cellcolor{gray!40} & \\
        \hline
        $v$  &  & \cellcolor{gray!40} \\
        \hline
    \end{tabular}\,.
\medskip

Let $f\in W_{m,n}$ and $\alpha\in K$. We say that $x\leftarrow x'-\alpha y$ (resp. $u\leftarrow u'-\alpha v$) is a change of variables to the right (resp. a downwards change of variables).
Indeed, these change of variables transform the table representation of $f$ from left to right and from top to bottom respectively.
Let $\alpha$, $\beta\in K$. We call renormalization the transformation $x\leftarrow\alpha x'$, $y \leftarrow \alpha^{-1} y'$, $u \leftarrow \beta u'$, $v \leftarrow \beta^{-1} v'$.
We note that all the above transformations can be seen as elements of $G$ and $H$. 
\bigskip

Let $f$ be a biform of bidegree $(m,n)$, with generic coefficients $(a_{ij})_{i,j}$. 
Let $C$ be a polynomial in $x,y$, $u,v$, and the coefficients of $f$, such that it is homogeneous in the coefficients of $f$, in $x,y$ and in $u,v$. 
We say that $C$ is a covariant of $f$ with respect to $H$ if for all $M = (M_1, M_2)\in H$, 
we have \[C(M\bigcdot (a_{ij})_{i,j}\,, M_1 (x,y)\,, M_2 (u,v)) = C((a_{ij})_{i,j}\,, (x, y), (u, v))\,.\]

Moreover, if $C$ does not contain any $x,y$ or $u,v$ terms, we say that $C$ is an invariant.
\medskip

Similarly, we define invariants and covariants with respect to $G$ for the elements of $W_m$.
$C$ is a covariant with respect to $G$ if it is a covariant with respect to
$H$, and if it satisfies with regard to the action of $\mathbb{Z}/2\mathbb{Z}$,
\[C((a_{ij})_{i,j}, (x,y), (u,v)) = C((a_{ji})_{i,j}, (u,v), (x,y))\,.\]

This paper contains the following results.
\begin{theorem*}[see Theorem~\ref{thm:secondary-invariants}]
    Let $K$ be an algebraically closed field of characteristic 0. Let $G$ be the group $\mathrm{SL}_2(K)\times \mathrm{SL}_2(K)\rtimes \mathbb{Z}/2\mathbb{Z}$, and $W_3$ be the space of biforms
    of bidegree $(3,3)$. Then the algebra $K[W_3]^{G}$ is generated by $65$ elements. Table~\ref{fig:tab_inv} contains the
    definition of these generators.
\end{theorem*}

\begin{cor*}
    An explicit classification of isomorphism classes of non-hyperelliptic curves of genus 4 of rank 4
    is given by Table~\ref{fig:tab_inv}.
\end{cor*}
\medskip

\begin{theorem*}[\protect{\cite[Table~6]{olive-gordan}}]
    Let $K$ be an algebraically closed field of characteristic 0, and $V_n$ be the space of binary forms of degree $n$ over $K$.
    The algebra $K[V_4\oplus V_6]^{\mathrm{SL}_2(K)}$ is generated by 60 invariants.
\end{theorem*}

\begin{cor*}
    An explicit classification of isomorphism classes of non-hyperelliptic curves of genus 4 of rank 3
    is given by \cite[Table~6]{olive-gordan}.
\end{cor*}

\section{Geometry of the problem}\label{sec:geometry}

Let $X,Y,Z,T$ be coordinates for $\mathbb{P}^3_K$. A curve will always be smooth, reduced and irreducible.
Let $\mathcal{C}$ be the canonical embedding in $\mathbb{P}^3_K$ of a non-hyperelliptic curve of genus 4 over $K$.
Then $\mathcal{C}$ is the complete intersection of a unique irreducible quadric and an irreducible cubic in $\mathbb{P}_{K}^3$~\cite[Example IV.5.2.2]{hartshorne}.
Moreover, the quadratic form that defines the quadric must be of rank 3 or 4.

Let $\mathcal{C}, \mathcal{C}'$ be two non-hyperelliptic curves of genus 4
canonically embedded in $\mathbb{P}_K^3$. The curve $\mathcal{C}$ is isomorphic
to $\mathcal{C}'$ if and only if there exists $f\in\mathrm{Aut}(\mathbb{P}^3_K)\simeq \mathrm{PGL}_4(K)$ such that $f(\mathcal{C})=\mathcal{C}'$.

\begin{prop}\label{prop:isomorphismP3}
    Let $Q, Q'\in K[X,Y,Z,T]$ (resp. $E,E'$) be irreducible homogeneous polynomials of degree 2 (resp.\ degree 3).
    Let $\mathcal{C}$ (resp.\ $\mathcal{C}'$) be the canonical embedding of the non-hyperelliptic curve of genus 4 
    defined over $K$ by $Q$ and $E$ (resp.\ $Q'$ and $E'$).
    Then the curves $\mathcal{C}$ and $\mathcal{C}'$ are isomorphic if and only if there exist $f\in \mathrm{GL}_4(K)$, and $\ell$ a linear form in $X,Y,Z,T$ such that \[
    \left\{\begin{array}{cl}
         f\bigcdot Q &= \hspace{.2cm}Q' \\
         f\bigcdot E &= \hspace{.2cm}\beta E'+\ell Q'\\
    \end{array}\right.\,.\]
\end{prop}

%
%

The action of the linear form is quite difficult to deal with, so we transform our problem so as to have a simpler action.
In order to do so, we study separately the cases of rank 3 and rank 4 separately. But first, we introduce a few tools from classical invariant theory.

\section{Preliminaries}\label{sec:tools}

In this section, we introduce the tools needed to study algebras of invariants.

\subsection{Hilbert series}

The Hilbert series of a graded algebra carries a great deal of combinatorial information about that algebra.

\begin{definition}
    Let $A$ be a finitely generated graded $K$-algebra, with \[\displaystyle A=\bigoplus_{d\geq 0}A_d\,.\]
    We define the Hilbert series of $A$ as the formal series \[HS(A,\, t) = \sum_{d = 0}^{+\infty}\mathrm{dim}(A_d)t^d\,.\]
\end{definition}

There are several ways of computing the Hilbert series of a graded algebra, one of them uses the theory of weights and roots.
This method is detailed in~\cite[Section 4.6]{kemper}.

\subsection{Linearly reductive groups}

\begin{definition}
    Let $\Gamma$ be a linear algebraic group. We say that $\Gamma$ is linearly reductive if
    for every rational representation $V$ of $\Gamma$, there exist $V_1,\ldots,V_r\subset V$ rational representations of $\Gamma$ such that
    $V = \bigoplus_{i=1}^rV_i$.
\end{definition}

Theorem~\ref{thm:hochster} shows that invariant algebras arising from the action of a linearly reductive group on a rational representation carry a nice structure.
We show in Proposition~\ref{prop:w3macaulay} that $H$ and $G$ are linearly reductive groups. 
We refer the reader to~\cite[Section 2]{kemper} for more information on linearly reductive groups.
\medskip

In the following, we let $A$ be a finitely generated graded $K$-algebra with $A=\bigoplus_{d=0}^{+\infty}A_d$, such that $A_0=K$.

\begin{definition}
    Let $f_1, \ldots, f_r\in A$ be homogeneous elements. We say that $f_1,\ldots, f_r$ is a homogeneous system of parameters if:
    \begin{itemize}
        \item[$-$] $f_1,\ldots, f_r$ are algebraically independent,

        \item[$-$] $A$ is a finitely generated $K[f_1,\ldots,f_r]$-module.
    \end{itemize}
\end{definition}

It is well-known that for any finitely generated graded $K$-algebra $A$, there exists a homogeneous system of parameters~\cite[Corollary 2.4.8]{kemper}.

\begin{definition}
    We say that $A$ is Cohen-Macaulay if there exists a homogeneous system of parameters $f_1,\ldots, f_r\in A$ for $A$,
    such that $A$ is free over $K[f_1,\ldots, f_r]$.
\end{definition}

\begin{prop}[\cite{kemper}]\label{prop:macaulay}
  Let us assume that $A$ is Cohen-Macaulay. We let $f_1, \ldots, f_r\in A$ be
  a homogeneous system of parameters such that we have the decomposition
  \[A = \bigoplus_{i=1}^s g_iK[f_1, \ldots, f_r]\,,\] with $g_i\in A$
  homogeneous of degree $e_i$. Then, the Hilbert series of $A$ is equal to
    \[HS(A, t) = \frac{\sum_{i = 1}^{s}t^{e_i}}{\prod_{j = 1}^r(1-t^{d_j})}\,,\]
    where $d_j = \deg(f_j)$.
\end{prop}

\begin{rem}
If $A$ is Cohen-Macaulay, then for every homogeneous system of parameters $f_1, \ldots, f_r \in A$,
the algebra $A$ is a free $K[f_1, \ldots, f_r]$-module~\cite[Proposition 2.5.3]{kemper}. Hence, the knowledge of the Hilbert series of $A$, together with a homogeneous system of parameters $f_1,\ldots,f_r$ for $A$,
ensures that any generator of $A$ which does not belong to $K[f_1, \ldots, f_r]$ must have degree at most $\max_{1\leq i \leq s}(e_i)$~\cite[Algorithm 2.6.1]{kemper}.
\end{rem}

\begin{theorem}[\cite{hochster}]\label{thm:hochster}
    Let $\Gamma$ be a linearly reductive group, and $V$ be a rational representation of $\Gamma$.
    Then $K[V]^\Gamma$ is Cohen-Macaulay.
\end{theorem}

\subsection{Nullcone of a rational representation}

\begin{definition}
    Let $V$ be a rational representation of an algebraic group $\Gamma$ over $K$. We define the nullcone of $V$ as
    \[\mathcal{N}_V^\Gamma = \{f\in V~|~\forall I\in K[V]_{>0}^\Gamma,~I(f) = 0\}\,.\]
\end{definition}

\begin{lemma}[\protect{\cite[Lemma 2.4.5]{kemper}}]\label{lem:linred}
    Let $V$ be a rational representation of a linearly reductive group $\Gamma$. Let $f_1, \ldots, f_r$ be homogeneous elements of $K[V]^\Gamma$ such that $V(f_1, \ldots, f_r) = \mathcal{N}_V^\Gamma$.
    Then $K[V]^\Gamma$ is finitely generated as a $K[f_1, \ldots, f_r]$-module.
\end{lemma}

\begin{cor}\label{cor:hsopnullcone}
    Let $V$ be a rational representation of a linearly reductive group $\Gamma$. Let $f_1, \ldots, f_r$ be homogeneous elements of $K[V]^\Gamma$, where $r = \dim(K[V]^\Gamma)$. We assume that $V(f_1, \ldots, f_r) = \mathcal{N}_V^\Gamma$.
    Then $f_1, \ldots, f_r$ is a homogeneous system of parameters for $K[V]^\Gamma$.
\end{cor}

\begin{proof}
    We know that the transcendence degree of $K[V]^\Gamma$ is equal to the Krull dimension of $K[V]^\Gamma$ (we refer the reader to~\cite[Theorem 5.9]{kemper2}).
    Since $K[V]^\Gamma$ is a finitely generated $K[f_1, \ldots, f_r]$-module by Lemma~\ref{lem:linred}, the invariants $f_1, \ldots, f_r$ must be algebraically independent (otherwise the dimension of $K[V]^\Gamma$ would be smaller).
    It follows that $f_1, \ldots, f_r$ is a homogeneous system of parameters.
\end{proof}

\begin{definition}
    Let $\Gamma$ be a linearly reductive group, $V$ a rational representation of $\Gamma$. Let $\mathbb{G}_m$
    be the algebraic group $\mathrm{Spec}(K[t,t^{-1}])$.
    We define the $1$-parameter subgroups of $\Gamma$ to be the non-trivial morphisms of algebraic groups (morphisms of varieties which also preserve the group structure)
    $\lambda : \mathbb{G}_m\rightarrow \Gamma$.
\end{definition}

\begin{definition}
    Let $\Gamma$ be a linearly reductive group, $V$ a rational representation of $\Gamma$ of dimension $n$.
    For any $1$-parameter subgroup $\lambda$ of $\Gamma$, we pick a basis of $V$ on which the action of $\lambda$ is diagonal
    (\cite[
    Section 7.3.3]{procesi}). In this basis, we let $\mathrm{diag}(t^{m_1},\ldots,t^{m_n})$ be the matrix of $\lambda$,
    and for any $f = (f_1,\ldots,f_n)\in V$, we define
    \begin{displaymath}
      \mu(f,\lambda) = \min\{m_i~|~i\text{ such that }f_i \neq 0\}\,.
    \end{displaymath}
\end{definition}

\begin{prop}[Hilbert-Mumford criterion, \protect{\cite[Chapter 9]{dolgachev}}] \label{Hilbert-Mumford}
    Let $\Gamma$ be a linearly reductive group, $V$ a rational representation of $\Gamma$.
    Let $f\in V$. Then $f\in \mathcal{N}_V^\Gamma$ if and only if there exists a $1$-parameter subgroup $\lambda$ of $\Gamma$
    such that $\mu(f,\lambda) > 0$.
\end{prop}

Using this criterion, one can effectively compute the nullcone of a rational representation.

\subsection{Transvectants of biforms}

The operators we define in this paragraph allow us to encode the invariants and covariants of biforms in a very compact way,
and they provide an efficient algorithm for evaluating invariants on given biforms.
We will call these operators transvectants (for biforms).
\medskip

It seems they were first introduced by Turnbull (around 1920) and rediscovered by Olver (around 1960)
to construct a generating set of invariants for biforms of bidegree (1,1),
(2,1) and (2,2)~\cite{turnbull, olver}.
To the best of our knowledge, these are the only instances where these operators have been used to construct invariants and covariants of biforms.
In~\cite{turnbull}, Turnbull shows that there is a Gordan algorithm for biforms.
He gives some examples, and in particular discusses the case of bicubic forms:
to run Gordan's algorithm in this case would require a generating set of
invariants of binary forms of degree 12, which is not yet known (although we
have a good idea of what it should be, see \url{https://www.win.tue.nl/~aeb/math/invar.html}).
In addition, the linear diophantine systems have too many solutions, so it seems that this strategy cannot be implemented in its current version for bicubic forms.
\medskip

For a good reference on transvectants, invariants and covariants of binary forms, and in particular Gordan's algorithm,
we refer the reader to~\cite{olive}.

\begin{definition}
    Let $m,n,m',m'>0$, and $k,l\geq 0$. We define a transvectant operator $W_{m,n}\times W_{m',n'}\rightarrow W_{m+m'-2k, n+n'-2l}$: 
    \[(f,g)_{k,l}:=\sum_{\substack{0\leq i\leq k \\ 0\leq j \leq l}}(-1)^{i+j}\binom{k}{i}\binom{l}{j}\frac{\partial^{k+l} f}{\partial x^i\partial y^{k-i}\partial u^j\partial v^{l-j}}\frac{\partial^{k+l} g}{\partial x^{k-i}\partial y^{i}\partial u^{l-j}\partial v^{j}}\,,\]
    where $f,g$ are biforms of bidegrees $(m,n)$ and $(m',n')$ respectively. 
    Moreover, when $k=l$, we write instead $(f,g)_k$.
\end{definition}

\begin{prop}
    Let $f,g$ be biforms.
    For any $k,l \geq 0$, the map $(f,g)\mapsto (f,g)_{k,l}$ is $H$-equivariant.
    Furthermore, if $f\in W_m$, $g\in W_n$, and $k=l$, this map is $G$-equivariant.
\end{prop}

\begin{proof}
    We notice that $(f,g)_{k,l}$ is the trace of the composition of two $\Omega$-processes~\cite[Equation 10.28]{olver},
    one on the set of variables $x,y$ and the second one on $u,v$.
    It follows that $(f,g)\mapsto (f,g)_{k,l}$ is $H$-equivariant.

    For the second statement, suppose that $f\in W_m$, $g\in W_n$ and $k=l$.
    The action of $\mathbb{Z}/2\mathbb{Z}$ on $f$ and $g$ switches $(x,y)$ and $(u,v)$,
    but it can also be seen as the exchange of the coefficients $a_{i,j}$ and $a_{j,i}$ for all $i,j$, in both $f$ and $g$.
    This action on the coefficients clearly commutes with the transvectant, hence $(f,g)\mapsto (f,g)_k$ is also $\mathbb{Z}/2\mathbb{Z}$-equivariant.
\end{proof}

\section{Case of rank $4$}\label{sec:rank4}

Let $Q\in K[X,Y,Z,T]$ be a homogeneous irreducible quadratic form of rank 4.
Up to a change of variables, we can assume that $Q = XT-YZ$. In the following, $Q$ always denotes $XT-YZ$.
Let $\mathcal{C}$ be the canonical embedding in $\mathbb{P}^3_K$ of a non-hyperelliptic curve of genus 4 which lies on the quadric defined by $Q$.
Let $E\in K[X,Y,Z,T]$ be a homogeneous irreducible cubic form such that $\mathcal{C}= \{Q=0\}\cap\{E=0\}$.
\medskip

We consider the Segre embedding

\[\fonction{\varphi}{\mathbb{P}_K^1\times\mathbb{P}_K^1}{\mathbb{P}_K^3}{([x:y],[u:v])}{[xu:xv:yu:yv]}\,.\]

It is well-known that $\varphi$ is an isomorphism from $\mathbb{P}_K^1\times\mathbb{P}_K^1$ to the quadric defined by $Q$ in $\mathbb{P}^3_K$.
Hence we can see the curve $\mathcal{C}$ as the pullback of $\{E=0\}$ by $\varphi$ in $\mathbb{P}_K^1\times\mathbb{P}_K^1$, which is defined by a bicubic form in $x,y$ and $u,v$.
For the rest of the paper, we write a generic bicubic form as:
\[f = a_{33}x^3u^3+a_{32}x^3u^2v+a_{31}x^3uv^2+a_{30}x^3v^3+a_{23}x^2yu^3+a_{22}x^2yu^2v+\ldots+a_{00}y^3v^3\,.\]

\subsection{Reduction to an algebraic problem}

We might ask what conditions must be met on two biforms for the corresponding curves to be isomorphic.

\begin{prop}
Let $E, E'\in K[X,Y,Z,T]$ be homogeneous irreducible cubic forms. Let $\mathcal{C} = \{E = 0\}\cap \{Q = 0\}$ and $\mathcal{C}' = \{E' = 0\}\cap \{Q = 0\}$. 
Let $f,f'\in W_3$ be the pullbacks of $E$ and $E'$ by $\varphi$. 
Then the curves $\mathcal{C}$ and $\mathcal{C}'$ are isomorphic if and only if
there exists $g\in \mathrm{Aut}(\mathbb{P}^1_K\times \mathbb{P}^1_K)$ such that $g\bigcdot f = f'$, where we identify $\mathrm{Aut}(\mathbb{P}^1_K\times \mathbb{P}^1_K)$
with the pullback of the subgroup of $\mathrm{GL}_4(K)$ which stabilizes $Q$.
\end{prop}

\begin{proof}
    We use the characterization of Proposition~\ref{prop:isomorphismP3}.
    Let $E, E'\in K[X,Y,Z,T]$ be homogeneous irreducible cubic forms. Let $\mathcal{C} = \{E = 0\}\cap \{Q = 0\}$ and $\mathcal{C}' = \{E' = 0\}\cap \{Q = 0\}$. 
    Let $f,f'\in W_3$ be the pullbacks of $E$ and $E'$ by $\varphi$. 
    Let $U$ be the subgroup of $\mathrm{GL}_4(K)$ which preserves $Q$. We can identify its pullback by $\varphi$ with $\mathrm{Aut}(\mathbb{P}^1_K\times \mathbb{P}^1_K)$.
    Then the curves $\mathcal{C}$ and $\mathcal{C}'$ are isomorphic if and only if there exists an element $M\in U$ such that
    $M\bigcdot E = E'+\ell Q$.
    We note that for any linear form $\ell\in K[X,Y,Z,T]$, the cubic forms $E'+\ell Q$ are pulled back to $f'$.
    Hence $\mathcal{C}$ and $\mathcal{C}'$ are isomorphic if and only if there exists $g\in \mathrm{Aut}(\mathbb{P}^1_K\times \mathbb{P}^1_K)$ such that $g\bigcdot f = f'$.
    We thus obtain the desired result.
\end{proof}

Consequently, the problem of finding the isomorphism class of a non-hyperelliptic curve of genus 4 of rank 4 reduces to the study of the orbits of bicubic forms under the action of $\mathrm{Aut}(\mathbb{P}^1_K\times \mathbb{P}^1_K)$.

\begin{lemma}
We have $\mathrm{Aut}(\mathbb{P}_K^1\times \mathbb{P}_K^1)\simeq \mathrm{PGL}_2(K)\times \mathrm{PGL}_2(K)\rtimes \mathbb{Z}/2\mathbb{Z}$,
where the first and second $\mathrm{PGL}_2(K)$ act on the first and second $\mathbb{P}^1_K$ respectively, and $\mathbb{Z}/2\mathbb{Z}$ acts by exchanging the two $\mathbb{P}^1_K$.
\end{lemma}

\begin{proof}
    It is clear that $\mathrm{PGL}_2(K)\times \mathrm{PGL}_2(K)\rtimes \mathbb{Z}/2\mathbb{Z}$ is a subgroup of $\mathrm{Aut}(\mathbb{P}_K^1\times \mathbb{P}_K^1)$.
    We now prove the converse inclusin, with arguments involving divisors.
    \medskip

    Let $\psi$ be an automorphism of $\mathbb{P}_K^1\times \mathbb{P}_K^1$. It induces an automorphism $\tilde{\psi}$ of $\Pic(\mathbb{P}_K^1\times \mathbb{P}_K^1)$.

    Let $\pi_1, \pi_2$ be the projections $\mathbb{P}_K^1\times \mathbb{P}_K^1 \rightarrow \mathbb{P}_K^1$ on the first and second $\mathbb{P}_K^1$ factor respectively.
    Let $\alpha = \mathbb{P}_K^1\times \{[0:1]\}$ and $\beta = \{[0:1]\} \times \mathbb{P}_K^1$ be the pullbacks of the point at infinity by $\pi_2$ and $\pi_1$ respectively.
    It is clear that $\alpha$ and $\beta$ are effective divisors of $\Pic(\mathbb{P}_K^1\times \mathbb{P}_K^1)$.
    \medskip

    It is well-known that there is an isomorphism $\Pic(\mathbb{P}_K^1\times
    \mathbb{P}_K^1) \simeq \mathbb{Z}\times \mathbb{Z}$~\cite[Exercise A.2.1]{hindry}.
    The divisors $\alpha$ and $\beta$ are generators of $\Pic(\mathbb{P}_K^1\times \mathbb{P}_K^1)$, so we identify them as
    generators of $\mathbb{Z}\times \mathbb{Z}$.

    The group morphism $\tilde{\psi}$ acts linearly on $\alpha$ and $\beta$, with integer coefficients.
    Moreover, if $D$ is an effective divisor, $\tilde{\psi}(D)$ and $\tilde{\psi}^{-1}(D)$ are also effective.
    Only two $\tilde{\psi}$ satisfy such conditions: the identity, and the automorphism that switches $\alpha$ and $\beta$.
    \medskip

    Let us assume that $\tilde{\psi}$ acts trivially on $\Pic(\mathbb{P}_K^1\times \mathbb{P}_K^1)$.
    Then $\psi$ preserves the hyperplane sections of each factor. It follows that $\psi\in \mathrm{PGL}_2(K)\times \mathrm{PGL}_2(K)$.
    \medskip

    If $\psi$ does not act trivially on $\Pic(\mathbb{P}_K^1\times \mathbb{P}_K^1)$, we let $\mu\in\mathrm{Aut}(\mathbb{P}_K^1\times \mathbb{P}_K^1)$ such that $\mu$ exchanges the two factors $\mathbb{P}^1$.
    The automorphism $\tilde{\mu}$ it induces on $\Pic(\mathbb{P}_K^1\times \mathbb{P}_K^1)$ is the non-trivial involution.
    The action of $\widetilde{\mu \circ \psi}$ on $\Pic(\mathbb{P}_K^1\times \mathbb{P}_K^1)$ is trivial, so we conclude that $\mu \circ \psi\in \mathrm{PGL}_2(K)\times \mathrm{PGL}_2(K)$.

    This proves the converse inclusion.
\end{proof}

It is easier to study the action of $\mathrm{SL}_2(K)\times \mathrm{SL}_2(K)\rtimes \mathbb{Z}/2\mathbb{Z}$,
thus we wish to reduce to that case.

\begin{lemma}
    Let $\{I_j\}$ be a homogeneous generating set of invariants for $K[W_3]^G$, with $I_j$ of degree $d_j$.
    Let $\mathcal{C}$, $\mathcal{C}'$ be two non-hyperelliptic curves of genus 4 defined by $f, f'\in W_3$ respectively.
    Then $\mathcal{C}$, $\mathcal{C}'$ are isomorphic if and only if there exists $\lambda\in K^\times$ such that
    \begin{align}
        I_j(f') = \lambda^{d_j}I_j(f)\,.
    \end{align}
\end{lemma}

\begin{proof}
    Let us assume that $\mathcal{C}$ and $\mathcal{C}'$ are isomorphic. There exists an element 
    \[M := (M_1, M_2, \varepsilon)\in \mathrm{GL}_2(K)\times\mathrm{GL}_2(K)\rtimes \mathbb{Z}/2\mathbb{Z}\] such that
    $M\bigcdot f = f'$. We define \[M'=\left(\frac{1}{\det(M_1)}M_1, \frac{1}{\det(M_2)}M_2, \varepsilon\right)\,,\text{ and } \lambda = (\det(M_1)\det(M_2))^3\,.\] 
    Clearly we have $M'\bigcdot f = \lambda f'$. Hence, for all $j$ we have 
    \[I_j(f) = I_j(M'\bigcdot f) = I_j(\lambda f') = \lambda^{d_j}I_j(f')\,.\]
    \medskip

    The converse implication relies on the fact that $\{I_j\}$ is a generating set of $K[W_3]^G$, hence it separates the orbits which are not unstable.
    We note that curves defined by elements of the nullcone $\mathcal{N}_{W_3}^G$ are singular.
    Hence, since by assumption the curves $\mathcal{C}$ and $\mathcal{C}'$ are not singular, the biforms $f$ and $f'$ do not belong to the nullcone $\mathcal{N}_{W_3}^G$.
    It follows that if their invariants are the same, they are $G$-equivalent.
\end{proof}

To exhibit such a set of generators for the algebra $K[W_3]^G$, we divide the study into two steps: first we study $K[W_3]^H$,
then we look at the action of $\mathbb{Z}/2\mathbb{Z}$ over $K[W_3]^H$. The following lemma
guarantees that this approach is valid.

\begin{lemma}
    We have \[K[W_3]^{G} = \left(K[W_3]^{H}\right)^{\mathbb{Z}/2\mathbb{Z}}\,.\]
\end{lemma}

\begin{proof}
    This is an application of a more general result: let $V$ be a rational representation of an algebraic group $\Gamma=M\rtimes N$.
    Then $K[V]^\Gamma=\left(K[V]^M\right)^N$. This result comes from the straightforward use of the definitions of the objects involved.
\end{proof}

This strategy is motivated by the fact that $H$ is connected, whereas $G$ is not.
Thus, using the theory of weights, we can compute the Hilbert series of $K[W_3]^{H}$,
but not the Hilbert series of $K[W_3]^{G}$. However, the Hilbert series of $K[W_3]^{G}$ is computed by different means,
and is given in Equation~\eqref{eq:hilb2}.
\medskip

We exhibit a generating set of invariants for $K[W_3]^H$ in 4 steps:
\begin{enumerate}
    \item Computation of the Hilbert series $HS(K[W_3]^H, t)$,

    \item Computation of the nullcone $\mathcal{N}_{W_3}^H$,

    \item Search of a homogeneous system of parameters for $K[W_3]^H$,

    \item Use of the Hilbert series to find generators of $K[W_3]^H$.
\end{enumerate}

In fact, we note in Remark~\ref{rem:samehsop} that any homogeneous system of parameters for $K[W_3]^G$
is a homogeneous system of parameters for $K[W_3]^H$, which modifies steps 2 and 3. We instead compute $\mathcal{N}_{W_3}^G$,
in order to find a homogeneous system of parameters for $K[W_3]^G$.

\subsection{Hilbert series $HS(K[W_3]^H, t)$}

The method we used to compute the Hilbert series of $K[W_3]^H$ is detailed in~\cite[Section 4.6]{kemper}.
It relies on the theory of weights and the computation of residues of meromorphic functions.
\medskip

Clearly $K[W_3]^H$ is a graded algebra: we have $K[W_3]=K[a_{i,j}]_{0\leq i,j\leq 3}$, with weights of the variables equal to $1$.
We compute the Hilbert series of $K[W_3]^H$ relatively to this grading.
\medskip

We verify the hypotheses of~\cite[Theorem 4.6.3]{kemper}:
\begin{enumerate}
    \item $\mathrm{SL}_2(K)$ is connected, since it is the affine scheme of an integral ring.
    Hence $H$ is connected.

    \item $K$ is algebraically closed of characteristic 0.

    \item $W_3$ is a rational representation of $H$.
\end{enumerate}

\begin{theorem}
    The Hilbert series of $K[W_3]^H$ is 
    \begin{equation}\label{eq:hilb1}
        \frac{P_1(t)}{(1-t^2)(1-t^4)^2(1-t^6)^2(1-t^8)^2(1-t^{10})(1-t^{12})(1-t^{14})}\,,
    \end{equation}

    where
    \begin{align*}
        P_1(t) =~ & t^{58}+2t^{52}+3t^{50}+9t^{48}+16t^{46}+26t^{44}+46t^{42}+69t^{40}+96t^{38}+129t^{36}+\\
               & 159t^{34}+181t^{32}+195t^{30}+195t^{28}+181t^{26}+159t^{24}+129t^{22}+96t^{20}+69t^{18}+\\
               & 46t^{16}+26t^{14}+16t^{12}+9t^{10}+3t^8+2t^6+1.
    \end{align*}
\end{theorem}

\begin{proof}
We use the theory of roots and weights developed in~\cite[A.4 and A.5]{kemper}.

The roots of $H$ are $(1,0), (0,1), (-1,0), (0,-1)$.
We choose a set of simple roots \[\{\alpha_1=(1,0), \alpha_2=(0,1)\}\,.\] 
The fundamental weights associated to this system
of simple roots is \[\lambda_1 = (1/2,0), \lambda_2 = (0,1/2)\,.\]
Let $z_1, z_2$ to be the characters associated to $\lambda_1$, $\lambda_2$ respectively,
with the notation $z_i=z^{\lambda_i}$.

In the canonical basis $x^iy^{3-i}u^jv^{3-j}$, the action of the maximal torus consisting of the matrices
$(\mathrm{diag}(z_1, 1/z_1),\mathrm{diag}(z_2, 1/z_2))$ is diagonal, with diagonal elements being \[z_1^{3-2i}z_2^{3-2j} = z^{\frac{1}{2}(3-2i, 3-2j)}\,.\]
Hence the weights of the $H$-module $W_3$ are \[\lambda_{i,j} = \frac{1}{2}(3-2i,3-2j)\,.\]

By~\cite[Theorem~4.6.3]{kemper}, it follows that the Hilbert series of $K[W_3]^{H}$ is the coefficient of $z_1^0z_2^0$ in the series expansion of the rational function
\begin{align*}
    g_0(z, t) &:= \frac{(1-z^{\alpha_1})(1-z^{\alpha_2})}{\prod\limits_{0\leq i,j\leq 3}(1-z^{\lambda_{i,j}}t)}\\
    & = \frac{(1-z_1^2)(1-z_2^2)}{\prod\limits_{0\leq i,j\leq 3}(1-z_1^{3-2i}z_2^{3-2j}t)}\,.
\end{align*}

In~\cite[File \texttt{HilbertSeries.m}]{Git}, we perform the whole calculation of the Hilbert series.
Let $g_1(z_1, t)$ be the sum of the residues of $g_0(z_1, z_2, t)/z_2$ seen as a rational function in $z_2$ with poles inside the circle $\{\lvert z_2\rvert = 1\}$
(considering that $\lvert t\rvert < \lvert z_1\rvert = 1$).
Finally, we compute the sum of the residues of $g_1(z_1, t)/z_1$ seen as a rational function in $z_1$ with poles inside the circle $\{\lvert z_1\rvert = 1\}$ (considering that $\lvert t\rvert < 1$), and the result is the Hilbert series of $K[W_3]^H$.

After some simplifications, we get Equation~\eqref{eq:hilb1}.
\end{proof}

\subsection{A homogeneous system of parameters for $K[W_3]^{G}$}

We now turn to Step~2 of our strategy, the computation of the nullcone $\mathcal{N}_{W_3}^G$ and the search of a homogeneous system of parameters for $K[W_3]^G$.

\subsubsection{Relations between $K[W_3]^G$ and $K[W_3]^H$}

We show in this paragraph that any homogeneous system of parameters of
$K[W_3]^{G}$ is a homogeneous system of parameters of $K[W_3]^{H}$.

\begin{lemma}
    The algebra $K[W_3]^H$ is integral over $K[W_3]^{G}$.
\end{lemma}

\begin{proof}
    Let $\sigma : K[W_3]^H \rightarrow K[W_3]^H$ be the restriction of the morphism $\sigma(a_{ij})=a_{ji}$ (defined on $K[W_3]$) to $K[W_3]^H$. 
    Since $H$ is normal in $G$, $\sigma$ is well-defined. Moreover, $K[W_3]^G$ is the subalgebra of $K[W_3]^H$ fixed by $\sigma$.
    
    Let $I\in K[W_3]^H$.
    We have $I^2-(\sigma(I)+I)I+\sigma(I)I = 0$. In addition, since $\sigma$ is an involution, we get $\sigma(\sigma(I)+I) = \sigma(I)+I$ and $\sigma(\sigma(I)I) = \sigma(I)I$.
    Hence $I$ is a root of a unitary polynomial of degree 2 with coefficients in $K[W_3]^{G}$.

    It follows that $K[W_3]^H$ is integral over $K[W_3]^{G}$.
\end{proof}

\begin{cor}\label{cor:nullconeequal}
    We have $\mathcal{N}_{W_3}^G=\mathcal{N}_{W_3}^H$.
\end{cor}

\begin{proof}
    Clearly $K[W_3]^{G}\subset K[W_3]^H$, thus we get $\mathcal{N}_{W_3}^H\subset \mathcal{N}_{W_3}^{G}$.

    For the converse inclusion, let $f\in\mathcal{N}_{W_3}^{G}$, and $I\in K[W_3]^H$. We have \[I^2(f)-[\sigma(I)+I](f) I(f)+[\sigma(I)I](f) = 0\,.\]
    By definition of $f$, $[\sigma(I)+I](f) = [\sigma(I)I](f) = 0$. It follows that $I^2(f)=0$.
\end{proof}

\begin{rem}\label{rem:samehsop}
    It follows from Corollary~\ref{cor:nullconeequal} that any homogeneous system of parameters of $K[W_3]^G$ is a homogeneous system of parameters of $K[W_3]^H$.
    Hence finding a homogeneous system of parameters for $K[W_3]^G$ will kill two birds with one stone.
\end{rem}

\subsubsection{One-parameter subgroups of $G$}

\begin{lemma}
    The 1-parameter subgroups of $H$ are the subgroups of $H$ which are conjugate to
    \begin{equation}\label{eq:1param}
        \left\{\left(\begin{pmatrix}t^k &0\\0& t^{-k}\end{pmatrix}, \begin{pmatrix} t^l &0 \\ 0& t^{-l} \end{pmatrix}\right), t\in K^\times\right\},
    \end{equation}
    for some $(k,l)\in\mathbb{Z}^2\backslash\{(0,0)\}$.
\end{lemma}

\begin{proof}
    Let $\lambda : \mathbb{G}_m \rightarrow H$ be a morphism of algebraic groups.
    Writing in coordinates, we have $\lambda = (\lambda_1,\lambda_2)$. We can see
    $\lambda_1,\lambda_2$ as rational representations of $\mathbb{G}_m$.
    Since $\mathbb{G}_m$ is a torus, these representations are diagonal in some basis~\cite[
    Section 7.3.3]{procesi}.
    Furthermore, the diagonal elements of $\lambda_1,\lambda_2$ must be integral powers of $t$~\cite[Theorem 1 from Section 7.3.3]{procesi}.
    Therefore, the $1$-parameter subgroups of $H$ are given by Equation~\eqref{eq:1param}.
\end{proof}

\begin{cor}\label{cor:1param}
    The 1-parameter subgroups of $G$ are the subgroups of $G$ which are conjugate to those given by Equation~\eqref{eq:1param}.
\end{cor}

\begin{proof}
    The image of the connected component of the identity by a morphism of algebraic groups
    is included in the connected component of the identity of the image. Thus, since $\mathbb{G}_m$ is connected (it is the affine scheme of an integral ring),
    any morphism of algebraic groups $\mathbb{G}_m\rightarrow G$ has its image contained in $H$.
\end{proof}
\medskip

Now we have all the tools we need to characterize $\mathcal{N}_{W_3}^G$.
But prior to doing that, we study the algebra $K[W_2]^G$.
For computational reasons, it seems that this detour is mandatory.

\subsubsection{Interlude: a homogeneous system of parameters for $K[W_2]^G$}\label{sec:hsopw2}

Although $K[W_2]^G$ was well understood in the 20th century~\cite{turnbull, olver}, we provide the results we need in more modern language.
The knowledge of a homogeneous system of parameters for $K[W_2]^G$ greatly simplifies the proof of Theorem~\ref{thm:hsop}.
Moreover, our strategy is the same for the study of $K[W_3]^G$ and $K[W_2]^G$:

\begin{enumerate}[label=(\Roman*)]
    \item Computation of the nullcone,
    \item Exhibition of a set of representatives,
    \item Proof that we found a homogeneous system of parameters.
\end{enumerate}

Step II is essential: by exhibiting a set of representatives for the $G$-action,
the vanishing of some invariants can be more easily exploited, thus simplifying Step III.

\begin{prop}\label{prop:nullcone_w2}
    The nullcone $\mathcal{N}_{W_2}^G$ is composed of the forms which are $G$-equivalent to a biquadratic form with representation Table~\ref{tab:NullconeW2}.
    \tabcolsep=0.11cm
    \begin{figure}[h]
        \centering    \tabcolsep=0.11cm
        \renewcommand{\arraystretch}{1.2}
        \begin{tabular}{|>{\centering}m{.6cm}|>{\centering}m{.6cm}|>{\centering}m{.6cm}|>{\centering\arraybackslash}m{.6cm}|}
        \cline{2-4} \multicolumn{1}{c|}{}& $x^2$ & $xy$ & $y^2$\\
        \hline
        $u^2$ &   &  & \cellcolor{gray!40} \\
        \hline
        $uv$ &   &  & \cellcolor{gray!40} \\
        \hline
        $v^2$ &  & \cellcolor{gray!40} & \cellcolor{gray!40}\\
        \hline
        \end{tabular}
        \subcaption{Elements of $\mathcal{N}_{W_2}^G$}
        \label{tab:NullconeW2}
    \end{figure}
\end{prop}

\begin{proof}
    Let $f\in \mathcal{N}_{W_2}^G$.
    By Hilbert-Mumford criterion~\ref{Hilbert-Mumford}, there exists a $1$-parameter subgroup $\lambda$ of $G$ such that $\mu(f,\lambda)>0$. 
    We choose a basis in which the action of $\lambda$ is diagonal, as in Equation~\eqref{cor:1param}.
    
    If $k > 0$ then the change of variables
    $\left(\begin{smallmatrix}
        0&-1\\
        1&0
    \end{smallmatrix}\right)$,
     transforms
    $\left(\begin{smallmatrix}
        t^k&0\\
        0&t^{-k}
    \end{smallmatrix}\right)$
    into
    $\left(\begin{smallmatrix}
        t^{-k}&0\\
        0&t^{k}
    \end{smallmatrix}\right)$.

    Hence we assume without loss of generality that $k,l\leq 0$. 
    Since $\mu(f,\lambda)>0$, the conditions on the coefficients of $f$ correspond to representation tables \ref{tab:nullW2_1} and \ref{tab:nullW2_2}.

\tabcolsep=0.11cm

    \begin{figure}[h]
        \begin{subfigure}[h]{0.3\textwidth}
            \centering    \tabcolsep=0.11cm
            \renewcommand{\arraystretch}{1.2}
            \begin{tabular}{|>{\centering}m{.6cm}|>{\centering}m{.6cm}|>{\centering}m{.6cm}|>{\centering\arraybackslash}m{.6cm}|}
            \cline{2-4} \multicolumn{1}{c|}{}& $x^2$ & $xy$ & $y^2$\\
            \hline
            $u^2$ &   &  & \cellcolor{gray!40} \\
            \hline
            $uv$ &   &  & \cellcolor{gray!40} \\
            \hline
            $v^2$ &  & \cellcolor{gray!40} & \cellcolor{gray!40}\\
            \hline
            \end{tabular}
            \subcaption{Case $k\leq l\leq 0$}
            \label{tab:nullW2_1}
        \end{subfigure}
        \begin{subfigure}[h]{0.3\textwidth}
            \centering    \tabcolsep=0.11cm
            \renewcommand{\arraystretch}{1.2}
            \begin{tabular}{|>{\centering}m{.6cm}|>{\centering}m{.6cm}|>{\centering}m{.6cm}|>{\centering\arraybackslash}m{.6cm}|}
        \cline{2-4} \multicolumn{1}{c|}{}& $x^2$ & $xy$ & $y^2$\\
        \hline
        $u^2$ &   &  &  \\
        \hline
        $uv$ &   &  & \cellcolor{gray!40} \\
        \hline
        $v^2$ & \cellcolor{gray!40} & \cellcolor{gray!40} & \cellcolor{gray!40}\\
        \hline
        \end{tabular}
        \subcaption{Case $l\leq k\leq 0$}
        \label{tab:nullW2_2}
    \end{subfigure}
    \end{figure}

Up to exchanging $(x,y)$ and $(u,v)$, which is a transformation in the group $G$, these cases
are the same. Hence $f$ is $G$-equivalent to a biform with representation Table~\ref{tab:NullconeW2}.
\end{proof}

\begin{lemma}\label{lemma:class_eq_w2} 
    Let $q\in W_2$. Then $q$ is $G$-equivalent to one of the biquadratic forms given in
  Figure~\ref{fig:W2Classes}.
\end{lemma}

\begin{figure}[H]
    \centering    \tabcolsep=0.11cm
    \renewcommand{\arraystretch}{1.2}
  \begin{subfigure}[h]{0.3\textwidth}
    \begin{tabular}{|>{\centering}m{.6cm}|>{\centering}m{.6cm}|>{\centering}m{.6cm}|>{\centering\arraybackslash}m{.6cm}|}
      \cline{2-4} \multicolumn{1}{c|}{}
      & $x^2$ & $xy$ & $y^2$\\
      \hline
      $u^2$ &   &  & \cellcolor{gray!40} \\
      \hline
      $uv$ &  & \cellcolor{gray!40} & \cellcolor{gray!40}\\
      \hline
      $v^2$ &\cellcolor{gray!40} & \cellcolor{gray!40} & \cellcolor{gray!40} \\
      \hline
    \end{tabular}
    \subcaption{Type~I}
    \label{tab:W2TypeI}
  \end{subfigure}
  \ %
  \begin{subfigure}[h]{0.3\textwidth}
    \begin{tabular}{|>{\centering}m{.6cm}|>{\centering}m{.6cm}|>{\centering}m{.6cm}|>{\centering\arraybackslash}m{.6cm}|}
      \cline{2-4} \multicolumn{1}{c|}{}
      & $x^2$ & $xy$ & $y^2$\\
      \hline
      $u^2$ &  & 1\cellcolor{gray!40} &  \\
      \hline
      $uv$ & & & \cellcolor{gray!40}\\
      \hline
      $v^2$ & \cellcolor{gray!40}& \cellcolor{gray!40} &  \cellcolor{gray!40}\\
      \hline
    \end{tabular}
  \subcaption{Type~II}
  \label{tab:W2TypeII}
  \end{subfigure}
  \ %
  \begin{subfigure}[h]{0.3\textwidth}
    \begin{tabular}{|>{\centering}m{.6cm}|>{\centering}m{.6cm}|>{\centering}m{.6cm}|>{\centering\arraybackslash}m{.6cm}|}
      \cline{2-4} \multicolumn{1}{c|}{}
      & $x^2$ & $xy$ & $y^2$\\
      \hline
      $u^2$ &   & 1\cellcolor{gray!40} &  \\
      \hline
      $uv$ & 1\cellcolor{gray!40}  & & \cellcolor{gray!40}\\
      \hline
      $v^2$ & \cellcolor{gray!40}& \cellcolor{gray!40} & \cellcolor{gray!40} \\
      \hline
    \end{tabular}
  \subcaption{Type~III}
  \label{tab:W2TypeIII}
  \end{subfigure}
  \caption{Representatives of $W_2$ for the $G$-action}
  \label{fig:W2Classes}
\end{figure}

\begin{proof}
    Let $q\in W_2$, with coefficients $b_{ij}\in K$.
    If $b_{22} \ne 0$, we let $\alpha$ be a solution of $b_{22}\alpha^2-b_{21}\alpha+b_{20} = 0$.
    The change of variables $u\leftarrow u'-\alpha v$ makes $b_{20}$ vanish.
    After the change of variables $u\leftarrow v'$, $v\leftarrow -u'$ the new coefficient $b_{22}$ is $0$. 
    Hence we can always assume that $b_{22} = 0$.
    \medskip

    The case $b_{12}=b_{21}=0$ corresponds to Table~(\ref{tab:W2TypeI}).
    \medskip

    If exactly one of them is not 0 (let's say $b_{12}$ for the computations, the involution of $G$ that exchanges $(x,y)$ and $(u,v)$ allows us to assume so),
    then we can transform $b_{11}$ and $b_{02}$ into $0$ with the change of variables $x \leftarrow x'-\alpha y, u \leftarrow u'-\beta v$, where $\alpha = b_{02}/b_{12}$ and $\beta = b_{11}/(2b_{12})$.
    Then, with the renormalization $u \leftarrow \mu u', v \leftarrow 1/\mu v'$, where $\mu^2 = 1/b_{12}$, $b_{12}$ becomes $1$.
    This corresponds to Table~(\ref{tab:W2TypeII}).
    \medskip

    If none of $b_{12}$ and $b_{12}$ vanish, then with similar transformations we get Table~(\ref{tab:W2TypeIII}).
\end{proof}

\begin{rem}
    By increasing the number of orbits, we could have set more coefficients. We chose not to, for the sake of simplicity.
\end{rem}

\begin{theorem}[\protect{\cite[File \texttt{ProofHSOP.m}]{Git}}]\label{thm:w2}
    Let $q$ be a biquadratic form with generic coefficients $b_{ij}$. We define $J_2 = (q,q)_2$, $J_3 = ((q,q)_1, q)_2$, and $J_4 = ((q,q)_2,(q,q)_2)_2$.
    Then \[\mathcal{N}_{W_2}^G=V(J_2,J_3,J_4)\,.\]
\end{theorem}

This theorem is proved in~\cite[File \texttt{ProofHSOP.m}]{Git}, but we detail the proof for
the generic case, which corresponds to Table~(\ref{tab:W2TypeIII}).

\begin{proof}
    Let $q\in V(J_2,J_3,J_4)$. By Lemma~\ref{lemma:class_eq_w2}, we can write $q$ as one of the representatives given in Figure~\ref{fig:W2Classes}.
    For each of them, we take the unknown coefficients to be indeterminates to cover all possible $q$.
    For Table~(\ref{tab:W2TypeIII}), we get $q = \sum_{0\leq i,j\leq 2}b_{ij}x^iy^{2-i}u^jv^{2-j}$, where $b_{22} = b_{21} = b_{12} = 0$,
    and the other $b_{ij}$ are indeterminates.
    \medskip

    We let $I = (J_2(q), J_3(q), J_4(q))$ be an ideal in $K[b_{20}, b_{10}, b_{01}, b_{00}]$.
    \medskip

    The radical decomposition of $I$ gives two possibilities: either $q$ has representation table
    
    \tabcolsep=0.11cm
    \renewcommand{\arraystretch}{1.2}
    \begin{center}
        \begin{tabular}{|>{\centering}m{.6cm}|>{\centering}m{.6cm}|>{\centering}m{.6cm}|>{\centering\arraybackslash}m{.6cm}|}
        \cline{2-4}
        \multicolumn{1}{c|}{} 
        & $x^2$ & $xy$ & $y^2$\\
        \hline
        $u^2$ &   & 1\cellcolor{gray!40} &  \\
        \hline
        $uv$ & 1\cellcolor{gray!40}  & &\\
        \hline
        $v^2$ & \cellcolor{gray!40} & & \\
        \hline
        \end{tabular}\,,
    \end{center} in which case $q\in\mathcal{N}_{W_2}^G$, either we make the change of variables $x \leftarrow x'+b_{20}y, u \leftarrow u'-b_{20}v$,
    which transforms the representation table of $q$ into
    
    \tabcolsep=0.11cm
    \renewcommand{\arraystretch}{1.2}
    \begin{center}
        \begin{tabular}{|>{\centering}m{.6cm}|>{\centering}m{.6cm}|>{\centering}m{.6cm}|>{\centering\arraybackslash}m{.6cm}|}
        \cline{2-4}
        \multicolumn{1}{c|}{} 
        & $x^2$ & $xy$ & $y^2$\\
        \hline
        $u^2$ & & 1\cellcolor{gray!40} & \cellcolor{gray!40} \\
        \hline
        $uv$ & 1\cellcolor{gray!40}  & & \\
        \hline
        $v^2$ & & & \\
        \hline
        \end{tabular}\,.
    \end{center}

    In any case, $q\in\mathcal{N}_{W_2}^G$.
\end{proof}

\subsubsection{A homogeneous system of parameters for $K[W_3]^G$}
\label{sec:homog-syst-param}

%
%
%
%
%
%


\begin{prop}
    The nullcone of $\mathcal{N}_{W_3}^G$ is composed of the bicubic forms whose representation table is given in Figure~\ref{fig:nullW3}:
\end{prop}

        \begin{figure}[H]
        \centering    \tabcolsep=0.11cm
        \renewcommand{\arraystretch}{1.2}
        \begin{subfigure}[h]{0.35\textwidth}
        \begin{tabular}{|>{\centering}m{.6cm}|>{\centering}m{.6cm}|>{\centering}m{.6cm}|>{\centering}m{.6cm}|>{\centering\arraybackslash}m{.6cm}|}
        \cline{2-5}
        \multicolumn{1}{c|}{} 
        & $x^3$ & $x^2y$ & $xy^2$ & $y^3$\\
        \hline
        $u^3$   && & \cellcolor{gray!40} & \cellcolor{gray!40}\\
        \hline
        $u^2v$  && & \cellcolor{gray!40} & \cellcolor{gray!40} \\
        \hline
        $uv^2$  &&& \cellcolor{gray!40} & \cellcolor{gray!40} \\
        \hline
        $v^3$  &&& \cellcolor{gray!40} & \cellcolor{gray!40}   \\
        \hline
        \end{tabular}\label{fig:tab_W31}
        \caption{Type I}
    \end{subfigure}
    \begin{subfigure}[h]{0.35\textwidth}
        \begin{tabular}{|>{\centering}m{.6cm}|>{\centering}m{.6cm}|>{\centering}m{.6cm}|>{\centering}m{.6cm}|>{\centering\arraybackslash}m{.6cm}|}
        \cline{2-5}
        \multicolumn{1}{c|}{} 
        & $x^3$ & $x^2y$ & $xy^2$ & $y^3$\\
        \hline
        $u^3$ & &&& \cellcolor{gray!40} \\
        \hline
        $u^2v$ &&& \cellcolor{gray!40}   &  \cellcolor{gray!40} \\
        \hline
        $uv^2$ & &&\cellcolor{gray!40}   & \cellcolor{gray!40}  \\
        \hline
        $v^3$ && \cellcolor{gray!40}  & \cellcolor{gray!40} & \cellcolor{gray!40}    \\
        \hline
        \end{tabular}\label{fig:tab_W32}
        \caption{Type II}
    \end{subfigure}
    \caption{Elements of the nullcone $\mathcal{N}_{W_3}^G$}
    \label{fig:nullW3}
    \end{figure}

The proof is very similar to that of Proposition~\ref{prop:nullcone_w2}.

In order to find a homogeneous system of parameters for $K[W_3]^G$,
we argue in the same way as in Section~\ref{sec:hsopw2}.

\begin{lemma}\label{lem:class_eq_w3}
    Let $f\in W_3$. Then $f$ is $G$-equivalent to a bicubic form whose representation table is given in Figure~\ref{fig:equi_class_w3}.
\end{lemma}

\begin{figure}[H]
    \centering    \tabcolsep=0.11cm
    \renewcommand{\arraystretch}{1.2}
    \begin{subfigure}[h]{0.33\textwidth}
    \begin{tabular}{|>{\centering}m{.6cm}|>{\centering}m{.6cm}|>{\centering}m{.6cm}|>{\centering}m{.6cm}|>{\centering\arraybackslash}m{.6cm}|}
    \cline{2-5}
    \multicolumn{1}{c|}{} 
    & $x^3$ & $x^2y$ & $xy^2$ & $y^3$\\
    \hline
    $u^3$ & & & \cellcolor{gray!40}  &  \cellcolor{gray!40}  \\
    \hline
    $u^2v$ &  & \cellcolor{gray!40} & \cellcolor{gray!40} & \cellcolor{gray!40} \\
    \hline
    $uv^2$ & & \cellcolor{gray!40} & \cellcolor{gray!40} & \cellcolor{gray!40} \\
    \hline
    $v^3$ & \cellcolor{gray!40} & \cellcolor{gray!40} & \cellcolor{gray!40} & \cellcolor{gray!40} \\
    \hline
\end{tabular}
\caption{Orbit I}
\label{tab:equiw31}
\end{subfigure}
\begin{subfigure}[h]{0.33\textwidth}
\begin{tabular}{|>{\centering}m{.6cm}|>{\centering}m{.6cm}|>{\centering}m{.6cm}|>{\centering}m{.6cm}|>{\centering\arraybackslash}m{.6cm}|}
    \cline{2-5}
    \multicolumn{1}{c|}{} 
    & $x^3$ & $x^2y$ & $xy^2$ & $y^3$\\
    \hline
    $u^3$ &&& \cellcolor{gray!40}  &  \cellcolor{gray!40} \\
    \hline
    $u^2v$ &  & \cellcolor{gray!40}  &\cellcolor{gray!40}  & \cellcolor{gray!40} \\
    \hline
    $uv^2$ & 1\cellcolor{gray!40} & &  \cellcolor{gray!40}& \cellcolor{gray!40} \\
    \hline
    $v^3$ & & \cellcolor{gray!40}  & \cellcolor{gray!40} &  \cellcolor{gray!40} \\
    \hline
\end{tabular}
\caption{Orbit II}
\label{tab:equiw32}
\end{subfigure}
\begin{subfigure}[h]{0.3\textwidth}
\begin{tabular}{|>{\centering}m{.6cm}|>{\centering}m{.6cm}|>{\centering}m{.6cm}|>{\centering}m{.6cm}|>{\centering\arraybackslash}m{.6cm}|}
    \cline{2-5}
    \multicolumn{1}{c|}{} 
    & $x^3$ & $x^2y$ & $xy^2$ & $y^3$\\
    \hline
    $u^3$ & & &\cellcolor{gray!40}  &  \cellcolor{gray!40}\\
    \hline
    $u^2v$ & 1\cellcolor{gray!40} &  &\cellcolor{gray!40}  & \cellcolor{gray!40} \\
    \hline
    $uv^2$ &  &\cellcolor{gray!40} &\cellcolor{gray!40} & \cellcolor{gray!40}\\
    \hline
    $v^3$ &\cellcolor{gray!40} & \cellcolor{gray!40}& \cellcolor{gray!40}& \cellcolor{gray!40} \\
    \hline
\end{tabular}
\caption{Orbit III}
\label{tab:equiw33}
\end{subfigure}
\bigskip

\begin{subfigure}[h]{0.33\textwidth}
    \begin{tabular}{|>{\centering}m{.6cm}|>{\centering}m{.6cm}|>{\centering}m{.6cm}|>{\centering}m{.6cm}|>{\centering\arraybackslash}m{.6cm}|}
    \cline{2-5}
    \multicolumn{1}{c|}{} 
    & $x^3$ & $x^2y$ & $xy^2$ & $y^3$\\
    \hline
    $u^3$ &  & 1\cellcolor{gray!40}  &  & \cellcolor{gray!40} \\
    \hline
    $u^2v$ & & & \cellcolor{gray!40}  &\cellcolor{gray!40}   \\
    \hline
    $uv^2$ & \cellcolor{gray!40}  & \cellcolor{gray!40} & \cellcolor{gray!40} & \cellcolor{gray!40} \\
    \hline
    $v^3$ &\cellcolor{gray!40}  & \cellcolor{gray!40} & \cellcolor{gray!40} & \cellcolor{gray!40}  \\
    \hline
\end{tabular}
\caption{Orbit IV}
\label{tab:equiw34}
\end{subfigure}
\begin{subfigure}[h]{0.33\textwidth}
\begin{tabular}{|>{\centering}m{.6cm}|>{\centering}m{.6cm}|>{\centering}m{.6cm}|>{\centering}m{.6cm}|>{\centering\arraybackslash}m{.6cm}|}
    \cline{2-5}
    \multicolumn{1}{c|}{} 
    & $x^3$ & $x^2y$ & $xy^2$ & $y^3$\\
    \hline
    $u^3$ &   & 1\cellcolor{gray!40} & & \cellcolor{gray!40} \\
    \hline
    $u^2v$ & 1\cellcolor{gray!40} & &\cellcolor{gray!40}& \cellcolor{gray!40}  \\
    \hline
    $uv^2$ & \cellcolor{gray!40} &\cellcolor{gray!40} &\cellcolor{gray!40} &\cellcolor{gray!40} \\
    \hline
    $v^3$ &\cellcolor{gray!40} &\cellcolor{gray!40} &\cellcolor{gray!40} & \cellcolor{gray!40} \\
    \hline
\end{tabular}
\caption{Orbit V}
\label{tab:equiw35}
\end{subfigure}
\caption{Orbits of $W_3$ under the action of $G$}
\label{fig:equi_class_w3}
\end{figure}

\begin{proof}
    The proof is very similar to that of Lemma~\ref{lemma:class_eq_w2}.

    First, we observe that we can always ensure that $a_{33}=0$. Either it is already 0, or we can make $a_{30}$ vanish using $a_{33}$, and then use the change of variables $u' = v, v' = -u$.

    \begin{itemize}
        \item[$-$] We assume that $a_{32}=a_{23}=0$.

            \begin{itemize}

                \item[$-$] If $a_{31}$ is 0, we get Table~\ref{tab:equiw31}.

                \item[$-$] Otherwise, with one change of variables downwards, one to the right and a renormalization, we transform the coefficients $a_{30}$ and $a_{21}$ into 0, and $a_{31}$ into 1.
                This case corresponds to Table~\ref{tab:equiw32}.
            \end{itemize}

        \item[$-$] If exactly one of $a_{32}$ and $a_{23}$ vanishes, and after
          a downwards change of variables and a change of variables to the
          right, the table representation of $f$ is either Table~\ref{tab:equiw33} or~\ref{tab:equiw34}.

        \item[$-$] If we assume that none of $a_{32}$ and $a_{23}$ vanish,
          then with a renormalization, we make $a_{23}$, $a_{32}$ equal to 1. A downwards change of variables and one to the right enable us to make two more coefficients vanish, and we get Table~\ref{tab:equiw35}.
    \end{itemize}
\end{proof}

These simplifications are not enough for the proof of Theorem~\ref{thm:hsop}, we need one last lemma.

\begin{lemma}[\protect{\cite[File \texttt{ProofHSOP.m}]{Git}}]\label{lemma:w3}
    Let $f\in W_3$. If $(f,f)_2 = 0$ and $(f,f)_3 = 0$, then $f\in \mathcal{N}_{W_3}^G$.
\end{lemma}

\begin{proof}
    Let $f$ be a bicubic form with indeterminate coefficients $a_{ij}$. We
    assume that $(f,f)_2 = (f,f)_3 = 0$. In some cases, the radical decomposition
    of the ideal defined by these conditions can be used to show that $f$ belongs to the nullcone. 
    However, there are cases for which the radical decomposition is hard to interpret,
    which is why we introduce another covariant of $f$.

    \medskip
    Let $f$ be a generic bicubic form with indeterminate coefficients $a_{ij}$.
    Let $J_{1,3}=(f,f)_{1,3}$. This covariant is a binary form of degree four in $x,y$.
    Its fundamental invariants as a binary form of degree 4 are of degree $4$ and $6$ in the coefficients of $f$.
    A simple computation shows that these invariants are in the radical of the ideal defined by the coefficients of $(f,f)_2$ and $(f,f)_3$.
    Therefore, they must vanish.

    If follows that for any $f\in W_3$, $J_{1,3}$ belongs to the nullcone of binary quartics, so it has a triple root. We assume without loss of generality that $J_{1,3}$ is divisible by $y^3$.
    It means that its coefficients in $x^4, x^3y, x^2y^2$ vanish.
    \medskip

    We reason with a generic representative of each equivalence class of $W_3$ given in Figure~\ref{fig:equi_class_w3}.
    We can do that, since the only transformations on $x,y$ used in Lemma~\ref{lem:class_eq_w3} are to the right.
    The other conditions $(f,f)_2 = (f,f)_3 = 0$ clearly stay unchanged after these transformations.

    We now compute the radical decomposition of the ideal defined by $(f, f)_3$, the coefficients of $(f,f)_2$, and the coefficients of $x^4, x^3y, x^2y^2$ in $J_{1,3}$.

    For representation Tables~\ref{tab:equiw31}, \ref{tab:equiw32}, \ref{tab:equiw33}, \ref{tab:equiw34}, from the radical decomposition of $I$ it follows that $f\in\mathcal{N}_{W_3}^G$.

    Now let $f\in W_3$ be a generic representative of the orbit given by Table~\ref{tab:equiw35}.
    We compute the radical decomposition of the ideal defined by all the above conditions.
    Either we have $f\in\mathcal{N}_{W_3}^G$, or the change of variables $x\leftarrow x'+a_{31}/2y$, $u\leftarrow u'-a_{31}/2v$
    transforms $f$ in such a way that $f\in\mathcal{N}_{W_3}^G$.
\end{proof}

\begin{theorem}[\protect{\cite[File \texttt{ProofHSOP.m}]{Git}}]\label{thm:hsop}
    Let $I_2$, $I_{4,1}$, $I_{4,2}$, $I_{6,1}$, $I_{6,2}$, $I_{8,1}$, $I_{8,2}$, $I_{10}$, $I_{12}$, $I_{14}$
    be the invariants defined in Table~\ref{fig:tab_hsop}.
    Then
    \begin{displaymath}
      \mathcal{N}_{W_3}^G = V(I_2,\,I_{4,1},\,I_{4,2},\,I_{6,1},\,I_{6,2},\,I_{8,1},\,I_{8,2},\,I_{10},\,I_{12},\,I_{14}).
    \end{displaymath}
\end{theorem}

\begin{proof}
    Let $f\in V(I_2$, $I_{4,1}$, $I_{4,2}$, $I_{6,1}$, $I_{6,2}$, $I_{8,1}$, $I_{8,2}$, $I_{10}$, $I_{12}$, $I_{14})$, and $h = (f,f)_2$.
    \medskip

    Since the invariants $I_{4,1}, I_{6,1}, I_{8,1}$ vanish when evaluated on $f$, we get that $h\in \mathcal{N}_{W_2}^G$.
    Hence we choose a basis such that the table representation of $h$ is:
    
    \tabcolsep=0.11cm
    \renewcommand{\arraystretch}{1.2}
    \begin{center}\begin{tabular}{|>{\centering}m{.6cm}|>{\centering}m{.6cm}|>{\centering}m{.6cm}|>{\centering\arraybackslash}m{.6cm}|}
        \cline{2-4}
        \multicolumn{1}{c|}{} 
        & $x^2$ & $xy$ & $y^2$\\
        \hline
        $u^2$ & &  & \cellcolor{gray!40}   \\
        \hline
        $uv$ & &  & \cellcolor{gray!40}\\
        \hline
        $v^2$ && \cellcolor{gray!40} & \cellcolor{gray!40}   \\
        \hline
        \end{tabular}.
    \end{center}

    In the following, we assume that $a_{33}=0$. We show afterwards that we can always reduce to that assumption.
    \medskip

    Since the change of variables used to transform $f$ into a representative given in Lemma~\ref{lem:class_eq_w3} are only to the right and downwards,
    we only need to prove the theorem for generic representatives given in Figure~\ref{fig:equi_class_w3}.

    Each case is covered in~\cite[File \texttt{ProofHSOP.m}]{Git}, but we detail the proof of the
    most generic case, which corresponds to Table~\ref{tab:equiw35}.

    Let $f = \sum_{0\leq i,j\leq 3}a_{ij}x^{3-i}y^iu^{3-j}v^j$, where $a_{33} = a_{22} = a_{13} = 0$ and $a_{23} = a_{32} = 1$, and the other $a_{ij}$ are indeterminates.

    Let $I$ be the ideal generated by the 5 vanishing coefficients of $h$, together with $I_2, I_{4,2}$, $I_{6,2}, I_{8,2}, I_{10}, I_{12}, I_{14}$. We denote by Rad$(I)$ its radical.

    A computation shows that every coefficient of a monomial in $h$ belongs to Rad$(I)$, thus we get $h = 0$.
    By Lemma~\ref{lemma:w3} it follows that $f\in\mathcal{N}_{W_3}^G$.
    \medskip

    Now let us deal with the case $a_{33}\neq 0$. With a downwards change of variables, we transform $a_{30}$ into 0, and then the change of variables $u\leftarrow v', v\leftarrow -u'$ transforms $a_{33}$ into $0$.
    However, the last change of variables transforms the representation table of $h$ into:

    \begin{center}
        \tabcolsep=0.11cm
        \renewcommand{\arraystretch}{1.2}
        \begin{tabular}{|>{\centering}m{.6cm}|>{\centering}m{.6cm}|>{\centering}m{.6cm}|>{\centering\arraybackslash}m{.6cm}|}
            \cline{2-4}
            \multicolumn{1}{c|}{} 
            & $x^2$ & $xy$ & $y^2$\\
            \hline
            $u^2$ & &\cellcolor{gray!40}  & \cellcolor{gray!40}  \\
            \hline
            $uv$ & & &\cellcolor{gray!40}  \\
            \hline
            $v^2$  & & &\cellcolor{gray!40}  \\
            \hline
        \end{tabular}\,.
    \end{center}

Let $b_{12}$ be the coefficient of $xyu^2$ in $h$.
We assume that $b_{12}$ is not 0, otherwise it reduces to the case we already dealt with.

A downwards change of variables modifies the 0 coefficients of $h$, which is why we try to reduce $f$ as much as possible, without
doing downwards transformations.
Hence, with the same idea of proof as Lemma~\ref{lem:class_eq_w3}, $f$ belongs is $G$-equivalent to an element of the following orbits:
\tabcolsep=0.09cm
\begin{center}
    \tabcolsep=0.11cm
    \renewcommand{\arraystretch}{1.2}
    \begin{tabular}{|>{\centering}m{.6cm}|>{\centering}m{.6cm}|>{\centering}m{.6cm}|>{\centering}m{.6cm}|>{\centering\arraybackslash}m{.6cm}|}
        \cline{2-5}
        \multicolumn{1}{c|}{}     & $x^3$ & $x^2y$ & $xy^2$ & $y^3$\\
        \hline
        $u^3$ & & & \cellcolor{gray!40}  &  \cellcolor{gray!40}  \\
        \hline
        $u^2v$ &  & \cellcolor{gray!40} & \cellcolor{gray!40} & \cellcolor{gray!40} \\
        \hline
        $uv^2$ & & \cellcolor{gray!40} & \cellcolor{gray!40} & \cellcolor{gray!40} \\
        \hline
        $v^3$ & \cellcolor{gray!40} & \cellcolor{gray!40} & \cellcolor{gray!40} & \cellcolor{gray!40} \\
        \hline
    \end{tabular}
~
    \begin{tabular}{|>{\centering}m{.6cm}|>{\centering}m{.6cm}|>{\centering}m{.6cm}|>{\centering}m{.6cm}|>{\centering\arraybackslash}m{.6cm}|}
        \cline{2-5}
        \multicolumn{1}{c|}{} 
        & $x^3$ & $x^2y$ & $xy^2$ & $y^3$\\
        \hline
        $u^3$ &&& \cellcolor{gray!40}  &  \cellcolor{gray!40} \\
        \hline
        $u^2v$ &  & \cellcolor{gray!40}  &\cellcolor{gray!40}  & \cellcolor{gray!40} \\
        \hline
        $uv^2$ & 1\cellcolor{gray!40} & &  \cellcolor{gray!40}& \cellcolor{gray!40} \\
        \hline
        $v^3$ & \cellcolor{gray!40}  & \cellcolor{gray!40}  & \cellcolor{gray!40} &  \cellcolor{gray!40} \\
        \hline
    \end{tabular}
~
    \begin{tabular}{|>{\centering}m{.6cm}|>{\centering}m{.6cm}|>{\centering}m{.6cm}|>{\centering}m{.6cm}|>{\centering\arraybackslash}m{.6cm}|}
        \cline{2-5}
        \multicolumn{1}{c|}{} 
        & $x^3$ & $x^2y$ & $xy^2$ & $y^3$\\
        \hline
        $u^3$ & & &\cellcolor{gray!40}  &  \cellcolor{gray!40}\\
        \hline
        $u^2v$ & 1\cellcolor{gray!40} &  &\cellcolor{gray!40}  & \cellcolor{gray!40} \\
        \hline
        $uv^2$ & \cellcolor{gray!40}  & &\cellcolor{gray!40} & \cellcolor{gray!40}\\
        \hline
        $v^3$ &\cellcolor{gray!40} & \cellcolor{gray!40}& \cellcolor{gray!40}& \cellcolor{gray!40} \\
        \hline
    \end{tabular}
\end{center}

\begin{center}
    \tabcolsep=0.11cm
    \renewcommand{\arraystretch}{1.2}
    \begin{tabular}{|>{\centering}m{.6cm}|>{\centering}m{.6cm}|>{\centering}m{.6cm}|>{\centering}m{.6cm}|>{\centering\arraybackslash}m{.6cm}|}
        \cline{2-5}
        \multicolumn{1}{c|}{} 
        & $x^3$ & $x^2y$ & $xy^2$ & $y^3$\\
        \hline
        $u^3$ & & &\cellcolor{gray!40}  &  \cellcolor{gray!40}\\
        \hline
        $u^2v$ & 1\cellcolor{gray!40} &  &\cellcolor{gray!40}  & \cellcolor{gray!40} \\
        \hline
        $uv^2$ & \cellcolor{gray!40}  &1\cellcolor{gray!40} &\cellcolor{gray!40} & \cellcolor{gray!40}\\
        \hline
        $v^3$ &\cellcolor{gray!40} & \cellcolor{gray!40}& \cellcolor{gray!40}& \cellcolor{gray!40} \\
        \hline
    \end{tabular}
    ~
    \begin{tabular}{|>{\centering}m{.6cm}|>{\centering}m{.6cm}|>{\centering}m{.6cm}|>{\centering}m{.6cm}|>{\centering\arraybackslash}m{.6cm}|}
        \cline{2-5}
        \multicolumn{1}{c|}{} 
        & $x^3$ & $x^2y$ & $xy^2$ & $y^3$\\
        \hline
        $u^3$ &  & 1\cellcolor{gray!40}  &  & \cellcolor{gray!40} \\
        \hline
        $u^2v$ & &\cellcolor{gray!40} & \cellcolor{gray!40}  &\cellcolor{gray!40}   \\
        \hline
        $uv^2$ & 1\cellcolor{gray!40}  & \cellcolor{gray!40} & \cellcolor{gray!40} & \cellcolor{gray!40} \\
        \hline
        $v^3$ &\cellcolor{gray!40}  & \cellcolor{gray!40} & \cellcolor{gray!40} & \cellcolor{gray!40}  \\
        \hline
    \end{tabular}
    ~
    \begin{tabular}{|>{\centering}m{.6cm}|>{\centering}m{.6cm}|>{\centering}m{.6cm}|>{\centering}m{.6cm}|>{\centering\arraybackslash}m{.6cm}|}
        \cline{2-5}
        \multicolumn{1}{c|}{} 
        & $x^3$ & $x^2y$ & $xy^2$ & $y^3$\\
        \hline
        $u^3$ &  & 1\cellcolor{gray!40}  &  & \cellcolor{gray!40} \\
        \hline
        $u^2v$ & &\cellcolor{gray!40} & \cellcolor{gray!40}  &\cellcolor{gray!40}   \\
        \hline
        $uv^2$ & 1\cellcolor{gray!40}  & \cellcolor{gray!40} & \cellcolor{gray!40} & \cellcolor{gray!40} \\
        \hline
        $v^3$ &\cellcolor{gray!40}  & \cellcolor{gray!40} & \cellcolor{gray!40} & \cellcolor{gray!40}  \\
        \hline
    \end{tabular}
\end{center}

\begin{center}
    \tabcolsep=0.11cm
    \renewcommand{\arraystretch}{1.2}
    \begin{tabular}{|>{\centering}m{.6cm}|>{\centering}m{.6cm}|>{\centering}m{.6cm}|>{\centering}m{.6cm}|>{\centering\arraybackslash}m{.6cm}|}
        \cline{2-5}
        \multicolumn{1}{c|}{} 
        & $x^3$ & $x^2y$ & $xy^2$ & $y^3$\\
        \hline
        $u^3$ &   & 1\cellcolor{gray!40} & \cellcolor{gray!40} & \cellcolor{gray!40} \\
        \hline
        $u^2v$ & 1\cellcolor{gray!40} & &\cellcolor{gray!40}& \cellcolor{gray!40}  \\
        \hline
        $uv^2$ & \cellcolor{gray!40} &\cellcolor{gray!40} &\cellcolor{gray!40} &\cellcolor{gray!40} \\
        \hline
        $v^3$ &\cellcolor{gray!40} &\cellcolor{gray!40} &\cellcolor{gray!40} & \cellcolor{gray!40} \\
        \hline
    \end{tabular}\,.
\end{center}

All cases are covered in
\cite[File \texttt{ProofHSOP.m}]{Git}. 
We only detail the proof for the most generic representative, which corresponds to the last representation table above.

Let $I$ be the ideal generated by the 5 vanishing coefficients of $h$, together with $I_2$, $I_{4,2}$, $I_{6,2}$, $I_{8,2}$, $I_{10}$, $I_{12}$, $I_{14}$. We denote by Rad$(I)$ its radical.

A computation shows that $a_{30}b_{12}\in \mathrm{Rad}(I)$. However, we assumed that $b_{12} \neq 0$, hence we must have $a_{30}=0$.

With the change of variables $u=v', v=-u'$, we reduced to the case where $a_{33}$ vanishes and $h$ is of the form
\begin{center}\begin{tabular}{|>{\centering}m{.6cm}|>{\centering}m{.6cm}|>{\centering}m{.6cm}|>{\centering\arraybackslash}m{.6cm}|}
    \cline{2-4}
    \multicolumn{1}{c|}{} 
    & $x^2$ & $xy$ & $y^2$\\
    \hline
    $u^2$ & &  & \cellcolor{gray!40}   \\
    \hline
    $uv$ & &  & \cellcolor{gray!40}\\
    \hline
    $v^2$ && \cellcolor{gray!40} & \cellcolor{gray!40}   \\
    \hline
    \end{tabular}\,.
\end{center}
Hence we can always assume that $a_{33} = 0$
\end{proof}

\subsection{A generating set of invariants for $K[W_3]^G$}

Several of the following results can be extended to $W_n$ for a general $n$, but for simplicity we only state them for $n=3$.

\begin{prop}\label{prop:w3macaulay}
    The algebras $K[W_3]^G$ and $K[W_3]^{H}$ are Cohen-Macaulay.
\end{prop}

\begin{proof}
    We prove that $H$ and $G$ are linearly reductive.
    First, we know that $\mathrm{SL}_2(K)$ is reductive~\cite[
    Section~7.3.2]{procesi}.
    Furthermore, the (semi-direct) product of two reductive groups is a reductive group~\cite[Section~7.3.6, Theorem~1]{procesi}.
    Thus, since the characteristic of $K$ is 0, $H$ is linearly reductive~\cite[Section~7.3.6, Theorem~2]{procesi}.

    The corollary of~\cite[Section 7.3.6, Proposition 2]{procesi} guarantees that $G$ is linearly reductive.
    It is clear that $W_3$ is a rational representation of both $H$ and $G$.
    
    Thus, from Hochster-Roberts theorem~\cite{hochster}, it follows that $K[W_3]^G$ and $K[W_3]^H$ are Cohen-Macaulay.
\end{proof}

%

\begin{lemma}[\protect{\cite[Page 38]{sturmfels}}]\label{lemma:inv_sec}
    Let $\mu_{1}, \ldots, \mu_{10}$ be any homogeneous system of parameters of $K[W_3]^G$.
    Let $\eta_1,\ldots, \eta_N\in K[W_3]^H$. Then $\eta_1, \ldots, \eta_N$ is a set of secondary invariants with respect to $\mu_1, \ldots, \mu_{10}$ if and only if
    $\eta_1, \ldots, \eta_N$ forms a basis of the $K$-vector space $K[W_3]^H/(\mu_1, \ldots, \mu_{10})$.
\end{lemma}

The proof, while not detailed in~\cite[Theorem 2.3.1]{sturmfels}, is rather elementary.

\begin{prop}\label{prop:inv_sec1}
    Let $\mu_1,\ldots,\mu_{10}$ be a homogeneous system of parameters for $K[W_3]^{G}$.
    There exist $\eta_1,\ldots, \eta_N$ homogeneous elements of $K[W_3]^H$ such that \[K[W_3]^H = \bigoplus_{j = 1}^N\eta_jK[\mu_1,\ldots,\mu_{10}]\,,\]
    with $\sigma(\eta_i) = \pm\eta_i$ for all $1\leq i\leq N$.
\end{prop}

\begin{proof}
    By Lemma~\ref{lemma:inv_sec}, finding such a decomposition is equivalent to constructing a basis of the $K$-vector space $V := K[W_3]^H/(\mu_1, \ldots, \mu_{10})$
    with the properties needed. Since the number of secondary invariants is finite, $V$ is a finite-dimensional $K$-vector space.
    Let $e_1, \ldots, e_N$ be a homogeneous basis of $V$.
    Clearly the set $\{e_i+\sigma(e_i)\}_i\cup \{e_i-\sigma(e_i)\}_i$ generates $V$. 
    Let $\{\eta_j\}_{1\leq j \leq N'}$ be a subset of these elements which forms a basis of $V$.
    In addition, we have $\sigma(\eta_j) = \pm\eta_j$ for all $1\leq j\leq N'$.
\end{proof}

\begin{cor}
    With the same notation, we have \[K[W_3]^G = \bigoplus_{j\in J}\eta_jK[\mu_1,\ldots,\mu_{10}]\,,\]
    where $J = \{1\leq j \leq N'~\lvert~\sigma(\eta_j) = \eta_j\}$.
\end{cor}

\begin{proof}
    Let $\mathcal{R}:K[W_3]^H\rightarrow K[W_3]^G$ such that for any $I\in K[W_3]^H$, \[\mathcal{R}(I) = \frac{I+\sigma(I)}{2}\,.\]
    
    It is the linear projection of $K[W_3]^H$ onto $K[W_3]^G$. For any $I_1\in K[W_3]^G$, $I_2\in K[W_3]^H$, we have $\mathcal{R}(I_1I_2)=I_1\mathcal{R}(I_2)$~\cite[Proposition 2.2.7]{kemper}.
    Hence, since $\mathcal{R}$ is surjective (it is a linear projection), and $\mu_1,\ldots, \mu_r$ are elements of $K[W_n]^G$, we have \[K[W_n]^{G} = \sum_{j = 1}^{N'}\mathcal{R}\left(\eta_j\right)K[\mu_1,\ldots,\mu_r].\]

    We naturally get \[K[W_n]^G = \sum_{j\in J}\eta_jK[\mu_1,\ldots,\mu_r]\,,\]
    where $J = \{1\leq j \leq N'~\lvert~\sigma(\eta_j) = \eta_j\}$.
    Moreover, this sum is direct, since \[K[W_n]^H = \bigoplus_{j = 1}^{N'}\eta_jK[\mu_1,\ldots,\mu_r]\,.\]
\end{proof}

\begin{cor}[\protect{\Cite[File \texttt{SecondaryInvariants.m}]{Git}}]\label{cor:inv_sec2}
    The Hilbert series of $K[W_3]^G$ is
    \begin{align}\label{eq:hilb2}
        \frac{P_2(t)}{(1-t^2)(1-t^4)^2(1-t^6)^2(1-t^8)^2(1-t^{10})(1-t^{12})(1-t^{14})}\,,
    \end{align}
    with
    \begin{align*}
        P_2(t) =~ & t^{52}+t^{50}+3t^{48}+6t^{46}+12t^{44}+23t^{42}+36t^{40}+51t^{38}+68t^{36}+84t^{34}+\\
        & 94t^{32}+99t^{30}+96t^{28}+87t^{26}+75t^{24}+61t^{22}+45t^{20}+33t^{18}+23t^{16}+\\
        & 14t^{14}+10t^{12}+6t^{10}+2t^8+t^6+1\,.
    \end{align*}
\end{cor}

\begin{proof}
    We exhibit a set of secondary invariants for $K[W_3]^H$ which satisfy the condition of Proposition~\ref{prop:inv_sec1}.
    We generate invariants degree by degree, until the dimension of the spanned vector space is equal to the dimension given by the Hilbert series given in Equation~\eqref{eq:hilb1}.
    Instead of working with the generic expression of the invariants themselves, we evaluate them on biforms, which is much more efficient.
    Our strategy is similar to that of~\Cite[Section 5.3]{olive-lerc}.
    \medskip

    Once this stage is complete, we can check which secondary invariants are fixed by $\sigma$, and which are not. Hence we get Equation~\eqref{eq:hilb2}.
\end{proof}

\begin{theorem}[\protect{\cite[File \texttt{SecondaryInvariants.m}]{Git}}]
  \label{thm:secondary-invariants}
    Let $K$ be an algebraically closed field of characteristic 0. Let $G := \mathrm{SL}_2(K)\times \mathrm{SL}_2(K)\rtimes \mathbb{Z}/2\mathbb{Z}$, and $W_3$ be the space of biforms
    of bidegree $(3,3)$. The algebra $K[W_3]^{G}$ is generated by $65$ elements.
\end{theorem}

This theorem follows from Corollary~\ref{cor:inv_sec2}: once the Hilbert series is known, we generate invariants degree by degree 
to find the dimension given by the Hilbert series.
A set of generating invariants is explicited in Table~\ref{fig:tab_inv}. 
Table~\ref{tab:deg_inv} shows the number of invariants of a given degree in the set of generating invariants.

\begin{figure}[h]
    \tabcolsep=0.15cm
    \renewcommand{\arraystretch}{1.2}
    \begin{tabular}{|c|c|c|c|c|c|c|c|c|c|}
        \hline
        Degree & 2 & 4 & 6 & 8 & 10 & 12 & 14 & 16 & 18\\
        \hline
        Number of fundamental invariants & 1 & 2 & 3 & 4 & 7 & 10 & 13 & 14 & 11\\
        \hline
    \end{tabular}
    \caption{Number of invariants of $K[W_3]^G$ by degree}
    \label{tab:deg_inv}
\end{figure}


\begin{rem}
    So far, we have only worked over algebraically closed fields. 
    In practice, however, one does not work with such fields. Let $K$ be field of characteristic 0 (not necessarily algebraically closed).
    
    Let $Q,E\in K[X,Y,Z,T]$ be an irreducible quadratic and cubic form, with $Q$ of rank $4$.
    There is an element $M(\Delta_1,\Delta_2)$ of $\mathrm{GL}_4(L)$ which maps $Q$ to $XT-YZ$, where $L=K(\Delta_1,\Delta_2)$ is (at most) a biquadratic
    extension of $K$. One can show that the invariants of the curve defined by $Q$ and $E$ belong to $K$, by considering the conjugate transformations $M(\pm \Delta_1, \pm\Delta_2)$.
\end{rem}

\subsection{Discriminant of a genus 4 curve of rank 4}

Once a model is chosen, the discriminant of a curve is a natural algebraic invariant. In the case of non-hyperelliptic genus 4 curves of rank 4,
one can compute the discriminant of a bicubic form of bidegree $(3,3)$ using~\cite{buse}.
It is an invariant of degree 34 in the coefficients of the bicubic form, hence not a fundamental invariant.

We use a strategy of evaluation interpolation to find the decomposition of the discriminant on the 65 invariants.
That decomposition can be found on~\cite[File \texttt{Discriminant.m}]{Git}.

\section{Case of rank $3$}\label{sec:rank3}

Let $K$ be an algebraically closed field of characteristic 0.
As in the case of the rank 4 quadric, we assume that the quadratic form is normalized, here $Q = XZ-Y^2$.
Let $\mathcal{C}$ be the canonical embedding of a non-hyperelliptic curve of genus 4 defined by $Q$ and an irreducible cubic form $E$.

Let \[\fonction{\varphi}{\mathbb{P}_K(1,1,2)}{\mathbb{P}^3_K}{[s:t:w]}{[s^2:st:t^2:w]},\]
where $\mathbb{P}_K(1,1,2)$ is the weighted projective space with weights $(1,1,2)$.

It is clear that $\varphi$ is an isomorphism from $\mathbb{P}_K(1,1,2)$ to $\{Q = 0\}$.
Therefore, as for the case of rank 4, we realize the curve $\mathcal{C}$ in $\mathbb{P}_K(1,1,2)$ as the pullback of the cubic $\{E=0\}$.
The pullback of this cubic is defined by a homogeneous polynomial of degree $6$ in $s,t,w$. We will write such a sextic generically as
\[f(s,t,w) = f_0w^3+f_2(s,t)w^2+f_4(s,t)w+f_6(s,t),\]
with $f_i\in K[s,t]$ homogeneous of degree $i$.

One can wonder what it means for curves to be isomorphic in such a space. Like
for the case of rank 4, the action of the linear form becomes trivial, and the subgroup of $\mathrm{PGL}_4(K)$ which preserves $Q$
is isomorphic to $\mathrm{\mathbb{P}_K(1,1,2)}$.

Moreover, we observe that $\mathbb{P}_K(1,1,2) = \mathrm{Proj}(K[s,t,w])$, with $s,t$ variables of weight 1, and $w$ of weight 2.
In \cite{alamrani}, it is shown that $\mathrm{Aut}(\mathbb{P}_K(1,1,2))\simeq \mathrm{Aut}(K[s,t,w])/K^\times$.

\begin{lemma}
    We have $\mathrm{Aut}(K[s,t,w])\simeq (\mathrm{GL}_2(K)\times K^\times)\ltimes K^3$, where
    $\mathrm{GL}_2(K)$ and $K^\times$ act naturally on $s,t$ and $w$ respectively, and $(a,b,c)\in K^3$ acts on $w$ as $w+as^2+bst+ct^2$.
\end{lemma}

\begin{proof}
    Let $\psi\in \mathrm{Aut}(K[s,t,w])$. Since $s$ and $t$ are the only elements of degree $1$,
    $\psi$ acts on $s,t$ via an element of $\mathrm{GL}_2(K)$.
    There only remains to know the image of $w$. Since $\psi(w)$ is a degree 2 element,
    it must be of the form $\alpha w+as^2+bst+ct^2$, with $\alpha\in K^\times$, and $a,b,c\in K$.

    Conversely, let $\psi$ be the $K[s,t,w]$ algebra morphism defined by
    $\psi(s,t)=M(s,t)$, with $M\in \mathrm{GL}_2(K)$, and $\psi(w)=\alpha w+as^2+bst+ct^2$,
    with $\alpha\in K^\times$, and $a,b,c\in K$. We can define an inverse for $\psi$:
    let $(s',t') = \psi(s,t)$ and $w'=\psi(w)$.
    We define $(s,t) = \psi^{-1}(s',t') = M^{-1}(s',t')$, and $w = \psi^{-1}(w') = 1/\alpha(w'-as^2-bst-ct^2)$.

    The action of $\mathrm{GL}_2(K)\times K^\times$ does not commute with the action of $K^3$.
    We get $\mathrm{Aut}(K[s,t,w])\simeq (\mathrm{GL}_2(K)\times K^\times)\ltimes K^3$.
\end{proof}

Let $f_0$ be the coefficient of $w^3$ in $f$. If $f_0$ is 0, the curve defined by $f$ has a singularity,
so we assume that $f_0 \neq 0$.
Up to rescaling, we have $f = w^3+f_2(s,t)w^2+f_4(s,t)w+f_6(s,t)$.
If we make the change of variables $w = w'-f_2(s',t')/3$, the term in $w^2$ vanishes,
and we get $f = w^3+f_4(s,t)w+f_6(s,t)$.

We call this form the canonical form of $f$. The only action which preserves a canonical form 
is the action of $\mathrm{GL}_2(K)$ on $s,t$.

Hence, we reduced to the problem of finding invariants of $V_4\oplus V_6$ under the action of $\mathrm{SL}_2(K)$,
where $V_n$ is the space of binary forms of degree $n$.
This problem was studied and solved by Olive.

\begin{theorem}[\protect{\cite[Table~6]{olive-gordan}}]
    The algebra $K[V_4\oplus V_6]^{\mathrm{SL}_2(K)}$ is generated by 60 invariants.
\end{theorem}

\begin{rem}
    In general, one does not work over algebraically closed fields.
    Let $K$ be a field of characteristic 0. Let $Q,E\in K[X,Y,Z,T]$ be irreducible homogeneous polynomials of degree 2 and 
    3 respectively, such that $Q$ is of rank 3. 
    Unlike the case of rank 4, the invariants of the curve defined by $Q$ and $E$ do not necessarily belong to $K$.
    They belong to (at most) a quadratic extension of $K$.
\end{rem}

%
%
%
%

\appendix
\section{Covariant tables}

\renewcommand{\arraystretch}{1.2}

\begin{table}[h]
    \begin{tabular}{|c|c|c|c|c|}
        \hline
        \diagbox[width=5.6em]{degree}{order} & 1 & 2 & 3 & 4\\
        \hline
        1 &  &  & $f$ & \\
        \hline
        2 &  & $h=(f,f)_2$ &  & $j=(f,f)_1$ \\
        \hline
        \multirow{2}{*}{3} & \multirow{2}{*}{$c_{31} = (h,f)_2$} &  & $c_{33,1} = (j,f)_2$ &  \\
        & & & $c_{33,2} = (h,f)_1$ & \\
        \hline
        \multirow{3}{*}{4} &  & $c_{42,1} = (h,h)_1$ &  & \multirow{2}{*}{$c_{44,1} = (c_{33,2},f)_1$} \\
        & & $c_{42,2} = (c_{31}, f)_1$ & & \multirow{2}{*}{$c_{44,2} = ((j,f)_1,f)_2$}\\
        & & $c_{42,3} = (c_{33,2}, f)_2$ & & \\
        \hline
        \multirow{3}{*}{5} & $c_{51, 1} = (c_{42,2}, f)_2$ &  & $c_{53, 1} = (c_{42,2}, f)_1$ &  \\
        & $c_{51,2} = (c_{44,1}, f)_3$& & $c_{53,2} = (c_{42, 3}, f)_1$ & \\
        & $c_{51,3} = (c_{44,2}, f)_3$ & & $c_{53,3} = ((f^3,f)_3, f)_3$ & \\
        \hline
        \multirow{3}{*}{6} &  & $c_{62,1} = (c_{53,1}, f)_2$ &  & \\
        & & $c_{62,2} = (c_{53,2}, f)_2$ & & \\
        & & $c_{62,3} = (c_{51,1}, f)_1$ & & \\
        \hline
        \multirow{3}{*}{7} & $c_{71,1} = (c_{62,1}, f)_2$ &  & \multirow{2}{*}{$c_{73,1} = (c_{62,2}, f)_1$} & \\
        & $c_{71,2} = (c_{51,1}, h)_2$ & & \multirow{2}{*}{$c_{73,2} = (c_{62,3}, f)_1$} & \\
        & $c_{71,3} = (c_{51,2}, h)_2$ & & & \\
        \hline
        \multirow{2}{*}{8} & & $c_{82,1} = (c_{71,1}, f)_1$ & & \multirow{2}{*}{$c_{84} = (c_{73,1}, f)_1$} \\
        & & $c_{82,2} = (c_{73,2}, f)_2$ & & \\
        \hline
        \multirow{2}{*}{9} & \multirow{2}{*}{$c_{91} = (c_{82,2}, f)_2$} & & $c_{93,1} = (c_{82,1}, f)_1$ & \\
        & & & $c_{93,2} = (c_{84}, f)_2$ & \\
        \hline
        10 & & $c_{102} = (c_{91}, f)_1$ & & \\
        \hline
        11 & $c_{111} = (c_{102}, f)_2$ & & $c_{113} = (c_{93,2}\cdot f, f)_3$ & \\
        \hline
    \end{tabular}
    \caption{Covariants required to compute the generators of $K[W_3]^G$}
    \label{fig:tab_cov}
\end{table}

\renewcommand{\arraystretch}{1.3}

\begin{table}
    \begin{tabular}{|c|c|}
        \hline
        Degree & Invariants \\
        \hline
        2 & $I_2 = (f,f)_3$ \\
        \hline
        \multirow{2}{*}{4} & $I_{4,1} = (h, h)_2$ \\
        & $I_{4,2} = (c_{33,1}, f)_3$ \\
        \hline
        \multirow{2}{*}{6} & $I_{6,1} = (c_{42,1}, h)_2$ \\
        & $I_{6,2} = (c_{53,3}, f)_3$ \\
        \hline
        \multirow{2}{*}{8} & $I_{8,1} = (c_{42,1}, c_{42,1})_2$ \\
        & $I_{8,2} = (c_{73,1}, f)_3$ \\
        \hline
        10 & $I_{10} = (c_{93,1}, f)_3$\\
        \hline
        12 & $I_{12} = (c_{113}, f)_3$\\
        \hline
        14 & $I_{14} = (c_{111}\cdot h, f)_3$\\
        \hline
    \end{tabular}
    \caption{A homogeneous system of parameters for $K[W_3]^G$}
    \label{fig:tab_hsop}
\end{table}

\renewcommand{\arraystretch}{1.3}

\begin{table}
    \begin{tabular}{|c|cc|}
        \hline
        Degree & \multicolumn{2}{c|}{Invariants} \\
        \hline
        2 & \multicolumn{2}{c|}{$I_2 = (f,f)_3$} \\
        \hline
        4 & $I_{4,1} = (h, h)_2$ & $I_{4,2} = (c_{33,1}, f)_3$ \\
        \hline
        \multirow{2}{*}{6} & $I_{6,1} = (c_{42,1}, h)_2$ & $I_{6,2} = (c_{53,3}, f)_3$ \\
        & \multicolumn{2}{c|}{$j_{6,1} = (c_{31}, c_{31})_1$} \\
        \hline
        \multirow{2}{*}{8} & $I_{8,1} = (c_{42,1}, c_{42,1})_2$ & $I_{8,2} = (c_{73,1}, f)_3$ \\
        & $j_{8,1} = (c_{31}, c_{51,1})_1$ & $j_{8, 2} = (c_{31}, c_{51,2})_1$\\
        \hline
        \multirow{4}{*}{10} & $I_{10} = (c_{93,1}, f)_3$ & $j_{10,1} = (c_{51,1}, c_{51,1})_1$ \\
        & $j_{10,2} = (c_{51,1}, c_{51,2})_1$ & $j_{10,3} = (c_{51,1}, c_{51,3})_1$ \\
        & $j_{10,4} = (c_{51,2}, c_{51,2})_1$ & $j_{10,5} = (c_{51,2}, c_{51,3})_1$\\
        & \multicolumn{2}{c|}{$j_{10,6} = (c_{51,3}, c_{51,3})_1$} \\
        \hline
        \multirow{5}{*}{12} & $I_{12} = (c_{113}, f)_3$ & $j_{12,1} = (c_{71,1}, c_{51,1})_1$ \\
        & $j_{12,2} = (c_{71,1}, c_{51,2})_1$ & $j_{12,3} = (c_{71,1}, c_{51,3})_1$ \\
        & $j_{12,4} = (c_{71,2}, c_{51,1})_1$ & $j_{12,5} = (c_{71,2}, c_{51,2})_1$ \\
        & $j_{12,6} = (c_{71,2}, c_{51,3})_1$ & $j_{12,7} = (c_{51,1}\cdot c_{31}^2, f)_1$ \\
        & $j_{12,8} = (c_{51,2}\cdot c_{31}^2, f)_1$ & $j_{12,9} = (c_{51,3}\cdot c_{31}^2, f)_1$ \\
        \hline
        \multirow{7}{*}{14} & $I_{14} = (c_{111}\cdot h, f)_3$ & $j_{14,1} = (c_{71,1}, c_{71,1})_1$ \\
        & $j_{14,2} = (c_{71,1}, c_{71,2})_1$ & $j_{14,3} = (c_{71,1}, c_{71,3})_1$ \\
        & $j_{14,4} = (c_{71,2}, c_{71,3})_1$ & $j_{14,5} = (c_{71,3}, c_{71,3})_1$ \\
        & $j_{14,6} = (c_{51,1}^2\cdot c_{31}, f)_3$ & $j_{14,7} = (c_{51,1}\cdot c_{51,2}\cdot c_{31}, f)_3$ \\
        & $j_{14,8} = (c_{51,1}\cdot c_{51,3}\cdot c_{31}, f)_3$ & $j_{14,9} = (c_{51,2}^2\cdot c_{31}, f)_3$ \\
        & $j_{14,10} = (c_{51,2}\cdot c_{51,3}\cdot c_{31}, f)_3$ & $j_{14,11} = (c_{51,3}^2\cdot c_{31}, f)_3$ \\
        & \multicolumn{2}{c|}{$j_{14,12} = (c_{71,1}\cdot c_{31}^2, f)_3$}\\
        \hline
        \multirow{7}{*}{16} & $j_{16,1} = (c_{71,1}\cdot c_{51,1}\cdot c_{31}, f)_3$ & $j_{16,2} = (c_{71,1}\cdot c_{51,2}\cdot c_{31}, f)_3$\\
        & $j_{16,3} = (c_{71,1}\cdot c_{51,3}\cdot c_{31}, f)_3$ & $j_{16,4} = (c_{71,2}\cdot c_{51,1}\cdot c_{31}, f)_3$ \\
        & $j_{16,5} = (c_{51,1}^3, f)_3$ & $j_{16,6} = (c_{51,1}^2\cdot c_{51,2}, f)_3$ \\
        & $j_{16,7} = (c_{51,1}^2\cdot c_{51,3}, f)_3$ & $j_{16,8} = (c_{51,1}\cdot c_{51,2}^2, f)_3$ \\
        & $j_{16,9} = (c_{51,1}\cdot c_{51,2}\cdot c_{51,3}, f)_3$ & $j_{16,10} = (c_{51,1}\cdot c_{51,3}^2, f)_3$ \\
        & $j_{16,11} = (c_{51,2}^3, f)_3$ & $j_{16,12} = (c_{51,2}^2\cdot c_{51,3}, f)_3$ \\
        & $j_{16,13} = (c_{51,2}\cdot c_{51,3}^2, f)_3$ & $j_{16,14} = (c_{51,3}^3, f)_3$ \\
        \hline
        \multirow{6}{*}{18} & $j_{18,1} = (c_{71,1}^2\cdot c_{31}, f)_3$ & $j_{18,2} = (c_{71,1}\cdot c_{71,2}\cdot c_{31}, f)_3$ \\
        & $j_{18,3} = (c_{71,2}^2\cdot c_{31}, f)_3$ & $j_{18,4} = (c_{71,1}\cdot c_{51,1}^2, f)_3$ \\
        & $j_{18,5} = (c_{71,1}\cdot c_{51,1}\cdot c_{51,2}, f)_3$ & $j_{18,6} = (c_{71,1}\cdot c_{51,1}\cdot c_{51,3}, f)_3$ \\
        & $j_{18,7} = (c_{71,1}\cdot c_{51,2}^2, f)_3$ & $j_{18,8} = (c_{71,1}\cdot c_{51,2}\cdot c_{51,3}, f)_3$ \\
        & $j_{18,9} = (c_{71,1}\cdot c_{51,3}^2, f)_3$ & $j_{18,10} = (c_{71,2}\cdot c_{51,1}\cdot c_{51,2}, f)_3$ \\
        & \multicolumn{2}{c|}{$j_{18,11} = (c_{71,2}\cdot c_{51,2}^2, f)_3$} \\
        \hline
    \end{tabular}
    \caption{A generating set of invariants for $K[W_3]^G$}
    \label{fig:tab_inv}
\end{table}

\clearpage

\printbibliography

\end{document}